\documentclass[12pt]{article}
\usepackage{amsfonts}
\usepackage{amsmath}
\usepackage{graphicx}
\usepackage{amssymb}
\usepackage{amsthm}
\usepackage{hyperref}
\usepackage[all]{xy}
\usepackage[english]{babel}

\title{Bordered Floer homology and a meridional class of knot}

\author{Jaepil Lee}

\newtheorem{mainthm}{Theorem}         
\newtheorem{maincor}[mainthm]{Corollary}       
\newtheorem{thm}{Theorem}[section]    
\newtheorem{lem}{Lemma}[section]      
\newtheorem{prop}[thm]{Proposition}
\theoremstyle{definition}
\newtheorem{defn}[thm]{Definition}    
\newtheorem{rem}[thm]{Remark}             

\makeatletter
  \let\c@lem=\c@thm
\makeatother

\begin{document}

\maketitle

\begin{abstract}   
For a knot $K$ and its knot Floer complex $CFK^-(K)$, we introduce an algorithm to compute the bordered Floer bimodule of the complement of the knot and its meridian. The grading of the module computes $spin^c$-summands of $\widehat{HFK}(S^3_{-n}(K), \mu_K)$, which can be also extended to arbitrary framing $n$.
\end{abstract}

\section{Introduction}
The classical Heegaard Floer homology package, introduced by Ozsv\'ath and Szab\'o~\cite{OZ02}, provides a 3+1 dimensional topological quantum field theory(TQFT) type invariants. These invariants were proven to be equivalent to 3-dimensional Seiberg-Witten invariants by Kutluhan, Lee, Taubes in~\cite{KLT10} and its four sequels, thus leading breakthrough in low-dimensional topology. The Heegaard Floer theory is also used to define knot and link invariants~\cite{OZ04, OZ05}, and especially the knot Floer homology can be used to find the Heegaard Floer homology of three-manifold obtained by an integral surgery on a knot~\cite{OZ08}. \\

The bordered Heegaard Floer homology was first introduced by Lipshitz, Ozsv\'ath, and Thurston in~\cite{LOT08}, and this package defines an invariant for three-manifold with a single boundary component. In particular, the boundary of three-manifold can be associated to a differential graded algebra, and the three-manifold can be associated to a module called a \emph{type-$D$ module} with the $dg$-algebra acting on the module. The torus boundary case was extensively studied because it was directly related to the knot Floer homology. In fact, the explicit algorithm has introduced in~\cite[Chapter 11]{LOT08}, which enables to find the structure of a $dg$-module of the knot complement from the knot Floer complex $CFK^-$. This technique was useful not only for the computation of $\widehat{HF}$ of a three-manifold of arbitrary integral Dehn surgery; but also the computation of the knot concordance invariant $\tau$ of cable knots~\cite{Hom13} and the concordance genus~\cite{Hom12}, and $L$-space classification problems~\cite{HL12}. There is a variant of this module named  \emph{type-$DD$ module}, invented for a manifold with two boundary components, also by Lipshitz, Ozsv\'ath and Thurston in~\cite{LOT11}. \\

On the other hand, the knot Floer homology on a knot embedded other than $S^3$ has drawn a particular interest. For a given knot $K \subset S^3$, the knot Floer homology of meridian in integral Dehn surgery manifold $S^3_{-n}(K)$ for sufficiently large integer $n$ attracted interests. In particular, Hedden studied the hat-version of knot Floer complex of meridian in the Dehn surgery manifold~\cite[Theorem 4.1]{Hed06} to compute the knot Floer homology of Whitehead double of a knot, and the infinity-version is given in~\cite[Theorem 4.2]{HKL12}. Both results only works for sufficiently large framing. Hedden, Kim, and Park recently studied some small framing cases~\cite{HKP17} as a part of the study on irreducible three-manifolds. \\

Inspired by these work, the main result of our paper enables the computation of the type-$D$ module of meridional class complement in $S^3_{-n}(K)$ from the knot Floer chain complex $CFK^-$. In what follows we shall outline the procedure to compute the type-$DD$ module of a link $\mathcal{L}_K$ complement, where link $\mathcal{L}$ is comprised of a knot $K \subset S^3$ and its meridian. \\
 
In order to describe the procedure, we will need to carefully choose sets of bases for $CFK^-(K)$. In fact, we have a horizontally or vertically simplified basis of $CFK^-(K)$~\cite[Definition 11.23]{LOT08}, such that the differential of every basis element is either zero or strictly drops Alexander filtration or $U$-filtration. For either basis, we can define a horizontal complex or a vertical complex;  these complexes are obtained by disregarding vertical arrows or horizontal arrows from $CFK^-(K)$, respectively. Then each complex has a unique distinguished element, which is a generator of $\widehat{HF}(S^3)$. Then we have the following result.
 
\begin{mainthm}
\label{thm:main}
Let $CFK^-(K)$ be a model for a reduced chain complex for a knot $K \subset S^3$. Then for a sufficiently large integer $n$, the type-$DD$ module of $S^3 \backslash \mathcal{L}_K$ with framing $-n$ can be derived from $CFK^-(K)$ by the following procedure. \\
Let $\{ \mathbf{x}^k \}$ be a vertically simplified basis with $\mathbf{x}^0$ being the distinguished element. For a vertical arrow of length $l$ from $\mathbf{x}^j$ to $\mathbf{x}^{j+1}$, the differential between the associated elements is
\begin{displaymath}
\xymatrix{
\mathbf{x}^j_0 \ar[rr]^{\rho_1 \sigma_3 + \rho_{123} \sigma_{123}} & &  \mathbf{x}^j_{\infty} \ar@/^/[r]^{\sigma_2} & \mathbf{x}^j_1 \ar@/^/[r]^{\sigma_{12} } \ar@/^/[l]^{\rho_{23} \sigma_1} & \cdots \ar@/^/[r]^{\sigma_{12} } \ar@/^/[l]^{\rho_{23}} & \mathbf{x}^j_l \ar@/^/[r]^{\sigma_1 } \ar@/^/[l]^{\rho_{23}} & \mathbf{x}^{j+1}_{\infty}  \ar@/^/[l]^{\rho_{23} \sigma_2} & & \mathbf{x}^{j+1}_0. \ar[ll]_{\rho_1 \sigma_3 + \rho_{123} \sigma_{123}}
}
\end{displaymath}
On the other hand, let $\{ \mathbf{y}^k \}$ be a horizontally simplified basis with $\mathbf{y}^0$ being the distinguished element. For a horizontal arrow of length $l$ from $\mathbf{y}^j$ to $\mathbf{y}^{j+1}$, the differential between the associated elements is
\begin{displaymath}
\xymatrix{
\mathbf{y}^j_{\infty} & &  \mathbf{y}^j_0 \ar[ll]_{\rho_1 \sigma_3 + \rho_{123} \sigma_{123}} \ar@/^/[r]^{\rho_3} & \mathbf{y}^j_{-1} \ar@/^/[r]^{\rho_{23} } \ar@/^/[l]^{\rho_2 \sigma_{12} } & \cdots \ar@/^/[r]^{\rho_{23} } \ar@/^/[l]^{\sigma_{12}} & \mathbf{y}^j_{-l} \ar@/^/[r]^{\rho_2 } \ar@/^/[l]^{\sigma_{12}} & \mathbf{y}^{j+1}_0 \ar@/^/[l]^{\rho_3 \sigma_{12}} \ar[rr]^{\rho_1 \sigma_3 + \rho_{123} \sigma_{123}} & & \mathbf{y}^{j+1}_{\infty}.
}
\end{displaymath}
Lastly, the \emph{unstable chain} between the two distinguished elements is as follows.
\begin{displaymath}
\xymatrix{
\mathbf{x}^0_{\infty} & & \mathbf{x}^0_0 \ar[ll]_{\rho_1 \sigma_3 + \rho_{123} \sigma_{123} } \ar@/^/[r]^{\rho_3} & \gamma_1 \ar@/^/[l]^{\rho_2 \sigma_{12} } \ar@/^/[r]^{\rho_{23}} & \cdots \ar@/^/[l]^{ \sigma_{12} } \ar@/^/[r]^{\rho_{23}} & \gamma_m \ar@/^/[l]^{ \sigma_{12} } \ar@/^/[r]^{\rho_{23} \sigma_1 } & \mathbf{y}^0_{\infty} \ar@/^/[l]^{\sigma_2} & & \mathbf{y}^0_0, \ar[ll]_{\rho_1 \sigma_3 + \rho_{123} \sigma_{123} }
}
\end{displaymath}
where $m = n + 2 \tau(K)$.
\end{mainthm}
It is crucial to mention that the same result can be achieved by Hanselman's versatile trimodule introduced in~\cite{Han13}. This trimodule is basically the bordered Floer invariant of the 3-link complement, where the link has three unknot components $U_1, U_2, U_3$ such that $\mathrm{lk}(U_1,U_2) = \mathrm{lk}(U_2,U_3) =1$ and $\mathrm{lk}(U_1, U_3)=0$. The derived tensor product with the trimodule and knot complement (using the trick introduced in~\cite[Section 2.3]{HL12} or~\cite[Section 2.3]{HRW16}) will produce the quasi-isomorphic bimodule but a different orientation and labelling from the convention used in this paper. However, the method used in this paper has the following advantage. First, our bimodule is reduced; i.e., there is no differential with algebra element 1. Second, the number of generators of the bimodule obtained by the algorithm is almost one third of the number of generators of the bimodule obtained from the trimodule; therefore the differential is less complicated. (The bimodule obtained from the trimodule has more generators, even after reducing all differentials of algebra element 1.) Third, the computation uses the pairing technique of the bordered Floer theory in a much relaxed condition. Recall that the original pairing theorem of~\cite{LOT08} dealt with the pairing of two domains satisfying certain conditions on the boundary chords~\cite[Chapter 3, Chapter 5]{LOT08}. However, the computation in this paper shows how these conditions can be dropped for certain cases. Especially, the moduli space of a boundary degeneration is paired with other domains, which was not used in the pairing of the original bordered Floer package. \\ 

More importantly, having a simplified type-$DD$ module also allows to sort the generators of a bimodule, and it may recover the filtration information of the knot Floer homology of non-classical knots; i.e., a knot embedded in three-manifold other than $S^3$. Taking the derived tensor product with the above type-$DD$ module and $\mathcal{A}_{\infty}$-module associated to the 0-surgery gives us the $\widehat{CFD}$ of meridian complement in $S^3_{-n} (K)$. Since the type-$D$ module is equipped with a grading depending on $spin^c$-structure,  sorting the module according to the grading reproves~\cite[Theorem 4.1]{Hed06}. Recall that the knot Floer homology $\widehat{CFK}(K) \cong \widehat{CF}(S^3)$ and is endowed with the Alexander filtration $\mathcal{F}(K,m)$.
\begin{mainthm}
Let $K \subset S^3$ be a knot with its meridian $\mu_K$, and let $n$ be a sufficiently large integer. Then, knot Floer homology $\widehat{HFK}_* (S^3_{-n} (K), \mu_K)$ can be decomposed as follows.
\begin{displaymath}
\widehat{HFK}_* (S^3_{-n}, \mu_K ) = \bigoplus_{m=-\frac{n}{2}}^{\frac{n}{2}-1} \left( H_* \left( \frac{\widehat{CF}(S^3)}{\mathcal{F}(K,m)} \right) \oplus H_* \left( \frac{\widehat{CF}(S^3)}{\mathcal{F}(K,-m-1)} \right) \right).
\end{displaymath}
(Here we are implicitly assuming $n$ is even; in the case $n$ is odd, the summation should be from $- \lfloor \frac{n}{2} \rfloor $ to $\lfloor \frac{n}{2} \rfloor$.)
\label{thm:sub}
\end{mainthm}
Although the result was stated for an arbitrary large negative framing $-n$, the computation used in this paper can be easily generalized to any integral framing by taking a derived tensor product with the type-$DA$ module of the toroidal mapping class group element $\tau_{\pm 1}$ given in~\cite[Section 10.2]{LOT11}. Tensoring with this module increases/decreases the framing by one, and eventually the type-$D$ module of a meridian complement in $S^3_{-n}(K)$ can be computed for any framing integer. Sorting out generators of the module by $spin^c$-structure grading will allow us to compute $\widehat{HFK}(S^3_n (K) , \mu_K)$ for arbitrary integer $n$. In particular, if $n= \pm 1$, then $S^3_{\pm 1}$ is an integral homology sphere and $\mu_K$ is null-homologous, so the Alexander grading $A$ of the knot Floer complex is well-defined. 

\begin{maincor}
Let $T$ be the left-handed trefoil knot. The knot Floer chain complex $\widehat{CFK}(S^3_{-1}(T), \mu_T)$ consists of three generators $x, y$ and $z$ such that $A(x)-A(y) = A(y) -A(z) =1$. 
\label{cor}
\end{maincor}

The example calculation in Section~\ref{sec:example} shows that it recovers not only the Alexander filtration of $\mu_T$ but also the $U$-filtration of $\mu_T$. Since the Kunnuth formula holds for the knot connected sum, this can lead to the double filtration information of a knot in the Poincare sphere $S^3_{-1}(T)$ and general integral homology spheres as well.

\subsection*{Organization}
In Section~\ref{sec:prelim} we overview the bordered Floer package for the torus boundary and discuss the algorithm to extract type-$D$ module of a knot complement from the knot Floer chain complex. Section~\ref{sec:diagram} considers the doubly bordered Heegaard diagram of a knot and its meridian complement, and Section~\ref{sec:main} computes the moduli space of holomorphic curves and proves Theorem~\ref{thm:main}. Section~\ref{sec:grading} computes the type-$D$ module of a merdian complement in $S^3_{-n}(K)$ and we see the collection of generators of the module with the same $spin^c$-grading is identical to the quotient of the knot Floer homology, thus proving Theorem~\ref{thm:sub}. In Section~\ref{sec:example} we prove Corollary~\ref{cor} with a model computation and discuss how to recover the knot filtration of a knot in an integral homology sphere.

\subsection*{Acknowledgement}
The author would like to thank Kyungbae Park for the helpful discussion. Also thanks to Robert Lipshitz for pointing out the relation between this work and Hanselman's trimodule~\cite{Han13}. Byungdo Park greatly helped revising the earlier version of this paper. Lastly, I would like to thank my advisor, Olga Plamenevskaya.

\section{A brief introduction of bordered Heegaard Floer homology}
\label{sec:prelim}
In this section we quickly recall definitions and properties of bordered Floer homology developed by Lipshitz, Ozsv\'ath, and Thurston. A more comprehensive account of the theory can be found in~\cite{LOT08, LOT11}. A more accessible introduction can also be found in~\cite{LOT11b}. We will merely list the essential part of their work which will be necessary for our purpose.

\subsection{Algebraic definitions} 
Let $(\mathcal{A},d)$ be a unital differential algebra over $\mathbb{F}_2$ with the subalgebra of idempotents $\mathcal{I} \subset \mathcal{A}$. $\mathcal{I}$ has a basis $\{ \imath_i \}$, such that $\imath_i \cdot \imath_j = \delta_{ij}$ and $\sum \imath_i = 1 \in \mathcal{A}$. \\

A \emph{(left) type-$D$ structure over $\mathcal{A}$} is an $\mathbb{F}_2$-module $N$ with left action of $\mathcal{I}$ and a map
\begin{displaymath}
\delta^1 : N \rightarrow \mathcal{A} \otimes_{\mathcal{I}} N
\end{displaymath}
satisfying the relation
\begin{displaymath}
( \mu \otimes \mathrm{id}_N ) \circ ( \mathrm{id}_{\mathcal{A}} \otimes \delta^1) \circ \delta^1 + (d \otimes \mathrm{id}_N ) \circ \delta^1 = 0,
\end{displaymath} 
where $\mu : \mathcal{A} \otimes \mathcal{A} \rightarrow \mathcal{A}$ is the multiplication of $\mathcal{A}$. \\

The above relation lets the tensor product $\mathcal{A} \otimes_{\mathcal{I}} N$ be a differential graded $\mathcal{A}$-module endowed with a structure $a \cdot (b \otimes x) = ab \otimes x$ and $\partial (a \otimes x) = a \cdot \delta^1(x) + d(a) \otimes x$. This module is called a \emph{type-$D$ module over $\mathcal{A}$}. \\

The map $\delta^1 : N \rightarrow \mathcal{A} \otimes_{\mathcal{I}} N$ can be extended to
\begin{displaymath}
\delta^k : N \rightarrow \mathcal{A}^{\otimes k} \otimes N
\end{displaymath}
defined by $\delta^k = (\mathrm{id}_{\mathcal{A}^{\otimes k-1}} \otimes \delta^1 )  \circ \delta^{k-1}$. If $k=0$, we let $\delta^0 := id_N$. The type-$D$ structure is called \emph{bounded} if $\delta^k =0$ for sufficiently large $k$. \\

An \emph{$\mathcal{A}_{\infty}$-module over $\mathcal{A}$}, or \emph{type-$A$ module} is an $\mathbb{F}_2$-module $M$ with a right action of $\mathcal{I}$ and family of maps
\begin{displaymath}
m_{i+1} : M \otimes_{\mathcal{I}} \mathcal{A}^{\otimes i} \rightarrow M
\end{displaymath}
satisfying the $\mathcal{A}_{\infty}$ relations
\begin{eqnarray*}
0 & = & \sum_{i=0}^n m_{n-i+1} ( m_{i+1} ( x \otimes a_1 \otimes \cdots \otimes a_i ) \otimes a_{i+1} \otimes \cdots \otimes a_n ) \\
 & + & \sum_{i=1}^{n-1} m_n ( x \otimes a_1 \otimes \cdots \otimes a_{i-1} \otimes \mu (a_i, a_{i+1} ) \otimes a_{i+1} \otimes \cdots \otimes a_n ) \\
 & + & \sum_{i=1}^n m_{i+1} ( x \otimes a_1 \otimes \cdots \otimes a_{i-1} \otimes d(a_i) \otimes a_{i+1} \otimes \cdots \otimes a_n).
\end{eqnarray*}
and unital conditions
\begin{eqnarray*}
m_2 (x, 1) & = & x \\
m_i (x, \cdots, 1, \cdots) & = & 0, \quad i>2. 
\end{eqnarray*}
If $m_k =0$ for sufficiently large $k$, we say the $\mathcal{A}_{\infty}$-module $M$ is \emph{bounded}. From now on, every module in this paper will be assumed to be bounded. \\

We will also need type-$A$ or type-$D$ modules of multiple right or left actions. In this paper we will focus on modules with two actions in the following sense. Let $(\mathcal{A}_1, d_1)$ and $(\mathcal{A}_2, d_2)$ be unital differential algebras over $\mathbb{F}_2$, with subalgebra of idempotents $\mathcal{I}_1$ and $\mathcal{I}_2$, respectively. A \emph{(left) type-$DD$ structure over $\mathcal{A}$} is an $\mathbb{F}_2$-module $N$ with left actions of $\mathcal{I}_1$ and $\mathcal{I}_2$, equipped with a map
\begin{displaymath}
\delta^1 : N \rightarrow ( \mathcal{A}_1 \otimes \mathcal{A}_2 ) \otimes_{\mathcal{I}_1 \otimes \mathcal{I}_2} N
\end{displaymath}  
satisfying a similar relation, so that the tensor product $(\mathcal{A}_1 \otimes \mathcal{A}_2) \otimes_{\mathcal{I}_1 \otimes \mathcal{I}_2} N$ becomes a differential module. This module is called a \emph{type-$DD$ bimodule over $\mathcal{A}_1 \otimes \mathcal{A}_2$}. Likewise, we say an $\mathbb{F}_2$-module $M$ is a \emph{(right) $\mathcal{A}_{\infty}$-bimodule over $\mathcal{A}_1 \otimes \mathcal{A}_2$} if $M$ is equipped with a right action of $\mathcal{I}_1 \otimes \mathcal{I}_2$ and family of maps 
\begin{displaymath}
m_{1 + i_1 + i_2} : M \bigotimes_{\mathcal{I}_1 \otimes \mathcal{I}_2} \mathcal{A}_1^{\otimes i_1} \otimes \mathcal{A}_2^{\otimes i_2} \rightarrow M
\end{displaymath}
satisfying the similar relation of $\mathcal{A}_{\infty}$-module. We remark that, for any given input, the relation should satisfy the sum of all terms which have a composition of two of $m$, $\mu$ and $d$ equal to zero. and refer the reader to~\cite{LOT11} for the explicit formulation of the relation. In~\cite{LOT11} authors also have defined type-$DA$ module, and the definition is similar to the type-$DD$ and type-$AA$ modules. A further generalization to a multimodule is found in~\cite{Han13}. 

\subsection{Torus algebra}
The bordered Floer homology package associates a boundary of a three-manifold to an algebra called \emph{strands algebra}. In particular, if the three-manifold has a torus boundary, then the algebra is called a \emph{torus algebra}, is written as $\mathcal{A}(T^2)$. The torus algebra is an $\mathbb{F}_2$-module generated by
\begin{displaymath}
\imath_1, \imath_2, \rho_1, \rho_2, \rho_3, \rho_{12}, \rho_{23}, \rho_{123}.
\end{displaymath}
$\imath_1$ and $\imath_2$ are the idempotents generating $\mathcal{I}$, such that $\imath_1 + \imath_2 =1$ is the identity. These generators satisfy the following relations:
\begin{eqnarray*}
\imath_1 \rho_1 = \rho_1 \imath_2 = \rho_1, & \imath_2 \rho_2 = \rho_2 \imath_1 = \rho_2, & \imath_1 \rho_3 = \rho_3 \imath_2 = \rho_3, \\
\imath_1 \rho_{12} = \rho_{12} \imath_1 = \rho_{12}, & \imath_2 \rho_{23} = \rho_{23} \imath_2 = \rho_{23}, & \imath_1 \rho_{123} = \rho_{123} \imath_2 = \rho_{123},
\end{eqnarray*}
and
\begin{displaymath}
\rho_1 \rho_2 = \rho_{12}, \quad \rho_2 \rho_3 = \rho_{23}, \quad \rho_1 \rho_{23} = \rho_{12} \rho_3 = \rho_{123}.
\end{displaymath}
In general, a strands algebra has a nontrivial differential, but the torus algebra has vanishing differential. \\

Every strands algebra will hereafter refer to the torus algebra, thus $\mathcal{A} (T^2)$ will be abbreviated to $\mathcal{A}$. 

\subsection{Bordered Heegaard diagram}
A \emph{bordered Heegaard diagram} for a three-manifold $Y$ with $\partial Y = T^2$ is a tuple $\mathcal{H} = ( \overline{\Sigma}, \overline{ \boldsymbol{\alpha} }, \boldsymbol{\beta}, z )$ consisting of the following:
\begin{itemize}
  \item a compact oriented surface $\overline{\Sigma}$ of genus $g$ with a single boundary $\partial \overline{\Sigma}$;
  \item $\overline{ \boldsymbol{\alpha} } = \{ \alpha^c_1 ,\cdots, \alpha^c_{g-1}, \alpha^a_1, \alpha^a_2 \}$, where $\alpha^c_i$ are pairwise disjoint circles in the interior of $\overline{\Sigma}$, and $\alpha^a_i$ are disjoint arcs on $\overline{\Sigma}$ away from $\alpha^c_i$ with endpoints on $\partial \Sigma$;
  \item a $g$-tuple of pairwise disjoint circles $\boldsymbol = \{ \beta_1, \cdots, \beta_g \}$ in the interior of $\overline{\Sigma}$; 
  \item a basepoint $z$ on $\partial \overline{\Sigma}$, away from $( \partial \alpha^a_1 ) \cup ( \partial \alpha^a_2 )$.
\end{itemize}
We require every intersection between $\alpha$-curve and $\beta$-curve to be transverse, and $\overline{\Sigma} \backslash \overline{ \boldsymbol{\alpha} }$ and $\overline{\Sigma} \backslash \boldsymbol{\beta}$ to be connected. From the diagram $\mathcal{H}$, we get a three-manifold with boundary by attaching a three-dimensional two-handle to $\overline{\Sigma} \times I$ along the $\boldsymbol{\alpha}$- and $\boldsymbol{\beta}$-circles. Then the $\boldsymbol{\alpha}$-arcs have the parametrization of the torus boundary. \\

The torus algebra is given by the boundary of the diagram. The boundary $\partial \overline{\Sigma}$ has a point $z$, and four other points of the $\partial \alpha^a_1$ and $\partial \alpha^a_2$. Then $\partial \overline{\Sigma}$ is an oriented circle, which has three intervals that do not contain $z$. Label these intervals 1,2 and 3, respecting the orientation of $\partial \overline{\Sigma}$. A Reeb chord on $\partial \overline{\Sigma}$ that starts and end on the $\alpha^a_i$ corresponds to the algebra element $\rho_I$, where $I \in \{ 1, 2, 3, 12, 23, 123 \}$ is determined by the intervals travelled by the chord. If we regard the idempotents $\imath_i$ as the sum of constant chords at $\partial \alpha^a_i$, then the multiplication rule can be regarded as a concatenation of chords.

\subsection{Moduli spaces of curves}

Let $\Sigma$ be a bordered Heegaard surface without compactification. Also let $\mathfrak{S} (\mathcal{H})$ be the set of unordered $g$-tuples $\mathbf{x} = \{ x_1, \cdots, x_g \}$ which contains exactly one point on each $\beta$-curve and exactly one point on each $\alpha$-curve, and at most one point on each $\alpha$-arc. Then we consider the $J$-holomorphic curves from Riemann surfaces with boundary punctures to $\Sigma \times [0,1] \times \mathbb{R}$ satisfying appropriate boundary conditions~\cite[Chapter 5]{LOT08}. Briefly, let $J$ be an admissible  almost complex structure on $\Sigma \times [0,1] \times \mathbb{R}$~\cite[Definition 5.1]{LOT08} such that the projection map $\Sigma \times [0,1] \times \mathbb{R} \rightarrow \Sigma$ is holomorphic ($\Sigma$ as a Riemann surface). Then we discuss $J$-holomorphic curves
\begin{displaymath}
u : (S, \partial S) \rightarrow ( \Sigma \times [0,1] \times \mathbb{R}, (\boldsymbol{\alpha} \times \{ 1 \} \times \mathbb{R} ) \cup (\boldsymbol{\beta} \times \{ 1 \} \times \mathbb{R} ) ),
\end{displaymath}
where $S$ is a Riemann surface with boundary punctures. \\

There are three different types of boundary punctures in $\partial S$; namely $+$, $-$ and $e$. These names are given by the asymptotic behavior of $u$ near the punctures. Let $t : \Sigma \times [0,1] \times \mathbb{R} \rightarrow \mathbb{R}$ be the projection. Then the $t$-coordinate at a boundary puncture $p$ is asymptotic to one of the following: $+ \infty$, $-\infty$ and some real number. Then $p$ is called $+$, $-$ and $e$ puncture, respectively. Observe that for a projection $\pi_{\Sigma} : \Sigma \times [0,1] \times \mathbb{R} \rightarrow \Sigma$,
\begin{itemize}
  \item the image of a $+$ or $-$ puncture under $\pi_{\Sigma}$ is on the intersection between $\boldsymbol{\alpha}$- and $\boldsymbol{\beta}$- curves;
  \item the image of an $e$ puncture under $\pi_{\Sigma}$ is a sequence of (ordered set of) Reeb chords $\overrightarrow{ \boldsymbol{\rho} } : = (\rho_{I_1}, \cdots \rho_{I_k})$ on $\partial \overline{\Sigma}$. (In fact, each $\rho_{I_i} \in \overrightarrow{ \boldsymbol{\rho} }$ can be also a set of Reeb chords. However, in this paper, $\rho_{I_i}$ is always considered as a singleton set of chord.)
\end{itemize}
In this sense, for $\mathbf{x}, \mathbf{y} \in \mathfrak{S} (\mathcal{H})$, if a curve $u$ is asymptotic to points corresponding to $\{ x_1, \cdots, x_g \} = \mathbf{x}$ at $-$ puncture and asymptotic to points corresponding to $\{ y_1, \cdots, y_g \} = \mathbf{y}$ at $+$ puncture, then we say the curve $u$ is connecting from $\mathbf{x}$ to $\mathbf{y}$ (possibly adjacent to Reeb chords represented by $\overrightarrow{ \boldsymbol{\rho} }$). \\

For $\mathbf{x}, \mathbf{y} \in \mathfrak{S} ( \mathcal{H} )$, let $\pi_2 (\mathbf{x}, \mathbf{y})$ denote the homology class of such curves that connects $\mathbf{x}$ to $\mathbf{y}$. For a homology class $B \in \pi_2 (\mathbf{x}, \mathbf{y})$, $\widetilde{\mathcal{M}}^B  (\mathbf{x}, \mathbf{y} )$ denotes the moduli space of holomorphic curves in $B$. If $B$ is adjacent to a sequence of Reeb chords $\overrightarrow{ \boldsymbol{\rho} }$, then we write $\widetilde{\mathcal{M}}^B (\mathbf{x}, \mathbf{y} ; \overrightarrow{ \boldsymbol{\rho} } )$. Just as the standard Floer theory, there is an $\mathbb{R}$-action on $\Sigma \times [0,1] \times \mathbb{R}$ by translation on the $t$-coordinate. Taking quotient of the action, the reduced moduli space is written as
\begin{displaymath}
\mathcal{M}^B (\mathbf{x}, \mathbf{y} ; \overrightarrow{ \boldsymbol{\rho} } ) : = \widetilde{\mathcal{M}}^B (\mathbf{x}, \mathbf{y} ; \overrightarrow{ \boldsymbol{\rho} } ) / \mathbb{R}.
\end{displaymath}

Discussing the expected dimension and modulo two count of a moduli space is usually not easy; but in some cases they are well-understood just by studying the image of a holomorphic curve under the projection $\pi_{\Sigma} : \Sigma \times [0,1] \times \mathbb{R} \rightarrow \Sigma$. A \emph{region of $\mathcal{H}$} is a connected component of $\overline{\Sigma} \backslash ( \overline{\boldsymbol{ \alpha }} \cup \boldsymbol{\beta} )$. A \emph{domain} is a linear combination of regions with integral coefficient. Then under the projection $\Sigma \times [0,1] \times \mathbb{R} \rightarrow \Sigma$, a curve of a homology class $B \in \pi_2 (\mathbf{x}, \mathbf{y})$ gives a domain $D$. In particular, if $u$ is a holomorphic curve, then the coefficient of its domain must be nonnegative. The domain $D$ may be adjacent to $\partial \overline{\Sigma}$; in that case we have a nonempty sequence of Reeb chords $\overrightarrow{ \boldsymbol{\rho} }$.  \\

A \emph{quadrilateral or rectangular domain} will refer to the domain whose coefficients are all 0 or 1, and the shape of regions with nonzero coefficient is quadrilateral or rectangular. Likewise, an \emph{annular domain} is a domain with coefficients that are all 0 and 1, and the shape of the regions with nonzero coefficients is an annulus. A provincial domain is a domain that is not adjacent to the boundary. A \emph{periodic domain} is a domain that does not have a corner between $\boldsymbol{\alpha}$- and $\boldsymbol{\beta}$- curves (but it may have a corner between $\boldsymbol{\alpha}$-curves and $\partial \overline{\Sigma}$). The space of periodic domain is denoted by $\pi_2 (\mathbf{x}, \mathbf{x})$. \\

For a homology class $B \in \pi_2 (\mathbf{x}, \mathbf{y})$ and a sequence of Reeb chords $\overrightarrow{ \boldsymbol{\rho} }$, the expected dimension $\mathrm{ind} (B, \overrightarrow{ \boldsymbol{\rho} })$ of the moduli space $\widetilde{ \mathcal{M} }^B (\mathbf{x}, \mathbf{y} ; \overrightarrow{ \boldsymbol{\rho} } )$ can be deduced from the following formula in~\cite[Definition~5.61]{LOT08}.
\begin{displaymath}
\mathrm{ind} (B, \overrightarrow{\boldsymbol{\rho}} ) = e(B) + n_{\mathbf{x}} (B) + n_{\mathbf{y}} (B) + | \boldsymbol{ \overrightarrow{\rho}} | + \iota ( \overrightarrow{ \boldsymbol{\rho} } ).
\end{displaymath}
The terms appearing in the above formula are explained below.
\begin{itemize}
  \item $e(B)$ is the \emph{Euler measure of the domain $B$}. The Euler measure of a region is defined as its Euler characteristic minus 1/4 the number of its corners (intersections between $\alpha$-curves and $\beta$-curves, and $\alpha$-curves and $\partial \overline{\Sigma}$), and additive under union.
  \item Let $\mathbf{x} \in \mathfrak{S} ( \mathcal{H} )$.  $n_{\mathbf{x}}$ is the sum of average multiplicity of four regions of each $x_i \in \mathbf{x}$.
  \item $| \boldsymbol{ \overrightarrow{\rho}} | + \iota ( \overrightarrow{ \boldsymbol{\rho} } )$ is quite complicated in general, but this quantity is simplified for the torus algebra. Let $\overrightarrow{ \boldsymbol{\rho} } = ( \rho_{I_1}, \cdots, \rho_{I_k} )$. Then the quantity is
  \begin{displaymath}
  \frac{1}{2} k + \sum_{s<t} L (\rho_{I_s} ,\rho_{I_t} ) ,
  \end{displaymath}
  where $L (\rho_{I_s} ,\rho_{I_t})$ equals 
  \begin{displaymath}
  L(\rho_{I_s}, \rho_{I_t}) = \left\{
    \begin{array}{ll}
      1/2 & \textrm{if $(I_s, I_t) = (1,2), (2,3), (12,3), (1,23)$} \\
      -1/2 & \textrm{if $(I_s, I_t) = (2,1), (3,2), (3,12), (23,1)$} \\
      1 & \textrm{if $(I_s, I_t) = (12,23)$} \\
      -1 & \textrm{if $(I_s, I_t) = (23,12)$} \\
      0 & \textrm{otherwise}
    \end{array}
    \right.
  \end{displaymath}
\end{itemize}

We now define the type-$D$ module $\widehat{CFD} (\mathcal{H})$. Let $X ( \mathcal{H} )$ be a $\mathbb{F}_2$-module spanned by $\mathfrak{S} (\mathcal{H})$, then
\begin{displaymath}
\widehat{CFD} ( \mathcal{H} ) := \mathcal{A} \otimes_{\mathcal{I}} X ( \mathcal{H} ),
\end{displaymath}
with $a \cdot (b \otimes \mathbf{x}) := (a \cdot b) \otimes \mathbf{x}$. Recall that the idempotent $\imath_i \in \mathcal{I}$ is associated to the arc $\alpha^a_i$. The idempotent action is defined as 
\begin{displaymath}
\imath_i \cdot \mathbf{x} := \left\{
\begin{array}{ll}
\mathbf{x} & \textrm{if $\mathbf{x}$ is \emph{not} occupying $\alpha^a_i$} \\
0 & \textrm{if $\mathbf{x}$ is occupying $\alpha^a_i$}.
\end{array}
\right.
\end{displaymath}

If a source of holomorphic curve $u$ having an $e$ puncture, we need to consider the contribution of the element of $\mathcal{A}$ to the differential of the type-$D$ module. Thus, for a sequence $\overrightarrow{ \boldsymbol{\rho} } = (\rho_{I_1}, \cdots, \rho_{I_k})$, we let
\begin{displaymath}
a( \overrightarrow{ \boldsymbol{\rho} } ) := \rho_{I_1} \cdots \rho_{I_k}.
\end{displaymath}
In other words, $a( \overrightarrow{ \boldsymbol{\rho} } )$ is merely a multiplication of all Reeb chords appearing in the sequence $ \overrightarrow{ \boldsymbol{\rho} } $. We also need to consider the case where the orientation of the boundary being reversed from the induced orientation which we will write $- \overrightarrow{ \boldsymbol{\rho} } := ( \overline{ \rho }_{I_1}, \cdots, \overline{\rho}_{I_k} )$. Here $\overline{ \rho }_{I_i}$ is the same Reeb chord as $\rho_{I_i}$ but with a reversed orientation regarded as a chord in $- \partial \overline{\Sigma}$. For example, $\overline{\rho}_1 = \rho_3$, $\overline{\rho}_2 = \rho_2$, $\overline{\rho}_3 = \rho_1$, $\overline{\rho}_{12} = \rho_{23}$ and so on. \\

The differential of $\widehat{CFD} (\mathcal{H})$ is defined by counting number of points of a reduced moduli space whose multiplicity of the region containing $z$ equals zero. Precisely, 
\begin{displaymath}
\partial ( \mathbf{x} ) : = \sum_{ \mathbf{y} \in \mathfrak{S} (\mathcal{H}) } \sum_{ 
\begin{subarray}{c}
B \in \pi_2 (\mathbf{x}, \mathbf{y}) \\
\{ \overrightarrow{\rho} | \mathrm{ind}(B, \overrightarrow{\rho})=1 \}
\end{subarray}
 }  
\# \mathcal{M}^B ( \mathbf{x}, \mathbf{y} ; \overrightarrow{ \boldsymbol{\rho} } ) a ( - \overrightarrow{ \boldsymbol{\rho} } ) \cdot \mathbf{y},
\end{displaymath}
where the multiplicity of the domain of $B$ at $z$ equals zero. \\ 

Type-$A$ module $\widehat{CFA} (\mathcal{H})$ is defined similarly. It is a $\mathbb{F}_2$-module with the same generating set $\mathfrak{S}( \mathcal{H} )$ endowed with the idempotent action
\begin{displaymath}
\mathbf{x} \cdot \imath_i := \left\{
\begin{array}{ll}
\mathbf{x} & \textrm{if $\mathbf{x}$ is occupying $\alpha^a_i$} \\
0 & \textrm{if $\mathbf{x}$ is \emph{not} occupying $\alpha^a_i$},
\end{array}
\right.
\end{displaymath}
We have a $\mathcal{A}_{\infty}$-relation on $\widehat{CFA} (\mathcal{H})$ such that
\begin{displaymath}
m ( \mathbf{x}, \rho_{I_1}, \cdots, \rho_{I_k} ) := \sum_{\mathbf{y} \in \mathfrak{S} (\mathcal{H})} \sum_{ \begin{subarray}{c} 
B \in \pi_2 (\mathbf{x}, \mathbf{y}) \\
\{ \overrightarrow{\rho} | \mathrm{ind}(B, \overrightarrow{\rho})=1 \}
\end{subarray}
}
\left( \# \mathcal{M}^B ( \mathbf{x}, \mathbf{y} ; \overrightarrow{ \boldsymbol{\rho} } ) \right) \mathbf{y}.
\end{displaymath}
Again, the multiplicity of the domain of $B$ at $z$ equals zero. \\

Having two modules $\widehat{CFA}(\mathcal{H}_1)$ and $\widehat{CFD} (\mathcal{H}_2)$, the bordered Floer theory computes the classical hat-version of the Heegaard Floer homology as follows. Let $\mathcal{H}_1 \cup \mathcal{H}_2$ be the the boundary sum of two bordered Heegaard diagrams $\mathcal{H}_1$ and $\mathcal{H}_2$, such that $\alpha$-arcs with the same labellings are paired on the boundary of the diagrams. Then $\widehat{CF} (\mathcal{H}_1 \cup \mathcal{H}_2 )$ is a chain complex generated by generators in $\widehat{CFA} (\mathcal{H}_1) \otimes_{\mathcal{I}} \widehat{CFD} (\mathcal{H}_2)$, with differential $\partial^{\boxtimes}$ defined by
\begin{displaymath}
x \otimes y \mapsto \sum_{k=0}^{\infty} (m_{k+1} \otimes \mathrm{id}) (x \otimes \delta^k (y) ).
\end{displaymath}
In fact, this chain complex is quasi-isomorphic to the Heegaard Floer complex of three-manifold represented by Heegaard diagram $\mathcal{H}_1 \cup \mathcal{H}_2$~\cite[Theorem 1.3]{LOT08}. This paring is called the \emph{derived tensor product}, and in general a derived tensor product of a type-$A$ module $M$ and a type-$D$ module $N$ is denoted by $(M \boxtimes N, \partial^{\boxtimes})$. \\ 

In~\cite{LOT11} bordered Heegaard package has been generalized to three-manifolds with two boundaries. Each boundary can be associated to either type-$A$ or type-$D$ structures, thus resulting in type-$DD$, $DA$ and $AA$ modules. The definitions are almost identical to the single boundary case, except there exists two different algebras called left and right torus algebra (they are the same torus algebra with different names) and they have the same action as described above. See~\cite[Chapter 2]{LOT11}.

\subsection{Simplified bases for the knot Floer complex}
This subsection provides a brief survey on the classical knot Floer homology. The detailed explanation is found in \cite[Chapter 11]{LOT08}, but this subsection follows the concise description in~\cite[Section 2]{Hom13}. The special basis of the knot Floer complex called \emph{simplified basis}. An advantage of the simplified basis is that it enables one to easily extract the type-$D$ module of a knot complement from the knot Floer chain complex. \\
 
Let $\mathcal{H}_K := ( \Sigma_0, \boldsymbol{\alpha}_0, \boldsymbol{\beta}_0, z, w)$ be the classical doubly pointed genus $g$ Heegaard diagram of a knot $K$ in $S^3$. Let $\mathfrak{S}_K$ be a set of $g$-tuples of intersection points between $\boldsymbol{\alpha}$ and $\boldsymbol{\beta}$ circles where each $\boldsymbol{\alpha}$ and $\boldsymbol{\beta}$ circles are used exactly once. The chain complex $CFK^-(K)$ is freely generated by the generator set $\mathfrak{S}_K$ over $\mathbb{F}_2 [U]$. The differential is defined as
\begin{displaymath}
\partial \mathbf{x} := \sum_{y \in \mathfrak{S}_K} \sum_{ 
  \begin{subarray}{l}
    \phi \in \pi_2 (\mathbf{x}, \mathbf{y}) \\
    \textrm{ind}(\phi) =1
  \end{subarray} }
\# \mathcal{M}(\phi) U^{n_w (\phi ) } \cdot \mathbf{y}.
\end{displaymath}
This complex has a homological $\mathbb{Z}$-grading, called the \emph{Maslov grading} $M$, and a $\mathbb{Z}$-filtration called the \emph{Alexander filtration} $A$. The relative Maslov grading and the Alexander filtration is defined as follows:
\begin{displaymath}
M(\mathbf{x}) - M(\mathbf{y} ) = \textrm{ind}(\phi) - 2 n_w (\phi) \quad \textrm{and} \quad A(\mathbf{x}) - A(\mathbf{y}) = n_z (\phi) - n_w (\phi), 
\end{displaymath}
where $\phi \in \pi_2 (\mathbf{x}, \mathbf{y})$. Multiplication by $U$ shifts the Maslov grading and the Alexander filtration as follows:
\begin{displaymath}
M(U \cdot \mathbf{x}) = M(\mathbf{x}) -2 \quad \textrm{and} \quad A( U \cdot \mathbf{x}) = A(\mathbf{x}) -1.
\end{displaymath}
As usual, we let $\widehat{CFK}(K) : = CFK^-(K) / (U=0)$. The Maslov grading is normalized so that the generator of $H_*(\widehat{CFK}(K)) \cong \widehat{HF} (S^3) \cong \mathbb{F}_2$ is supported in Maslov grading zero. Alexander filtration carries over to $\widehat{CFK}(K)$ and conventionally the subgroup of $\widehat{CFK}(K)$ generated by Alexander filtration less than equal to $m$ will be denoted by $\mathcal{F}(K,m)$. \\

The Alexander grading $A$ is given as follows: The homology $\widehat{HFK}(K)$ of $\widehat{CFK}(K)$ has a decomposition
\begin{displaymath}
\widehat{HFK}(K) = \bigoplus_s \widehat{HFK}(K,s)
\end{displaymath}
where $s$ is the Alexander grading induced by the filtration. Here we use the normalize Alexander grading defined by the following relation:
\begin{displaymath}
\min \{ s |  \widehat{HFK}(K,s) \neq 0 \} = - \max \{ s |  \widehat{HFK}(K,s) \neq 0 \}.
\end{displaymath}

Finally, the complex $CFK^{\infty}(K) := CFK^-(K) \otimes_{\mathbb{F}_2 [U]} \mathbb{F}_2 [U, U^{-1}]$ is naturally a $\mathbb{Z} \otimes \mathbb{Z}$-filtered complex, with one filtration given by the $(-U)$-exponent and the other by the Alexander filtration. Traditionally, each element of the complex $CFK^{\infty}(K)$ is plotted on the $(i,j)$-plane, where the $i$-th coordinate is the $(-U)$-exponent and the $j$-th coordinate the Alexander grading; i.e., the element $U^i \cdot \mathbf{x}$ will be plotted at the coordinate $(-i, A(U^i \cdot \mathbf{x}))$. Thus the differential $\partial$ of the complex is depicted as an arrow pointing (non-strictly) downwards and to the left. \\

The complex of $CFK^{\infty}(K)$ has a $\mathbb{Z} \oplus \mathbb{Z}$-filtration, and a subset $S$ of $\mathbb{Z} \oplus \mathbb{Z}$ may be used to describe a subcomplex of $CFK^{\infty}(K)$. Let $C(S) \subset CFK^{\infty}(K)$ consists of points whose $(i,j)$-coordinates are in $S$. Although not every $C(S)$ is a subcomplex of $CFK^{\infty}(K)$, but for some appropriate $S$, $C(S)$ may inherit the quotient complex structure. For instance, $C(i=0)$ can be identified to $\widehat{CFK}(K)$. \\

In~\cite{OZ03}, the smooth concordance invariant $\tau(K)$ is defined as follows.
\begin{displaymath}
\tau(K) := \min \{ s \ | \ \imath : C(i=0, j \leq s) \rightarrow C(i=0) \textrm{ induces a nontrivial map on homology} \}
\end{displaymath}

Although \cite[Theorem 11.36]{LOT08} is originally stated in basis-free version, for our purpose we choose a specific basis for $CFK^{\infty}(K)$. Let $C^h := C(j=0)$ be a complex equipped with a differential $\partial^h$, which is called a \emph{horizontal complex}. We view this complex as a subquotient complex of $CFK^{\infty}(K)$ consisting of elements with $j$-coordinates equal to zero, with differential pointing towards (non-strictly) to the left. The horizontal complex inherits the $\mathbb{Z}$-filtration from $CFK^{\infty}(K)$ by $(-U)$-exponents. Likewise, we define $C^v : = C(i=0)$ equipped with a differential $\partial^v$, called a \emph{vertical complex}. The vertical complex also inherits the $\mathbb{Z}$-filtration structure from $CFK^{\infty}(K)$ by the Alexander filtration. \\

$CFK^{\infty}(K)$ is called \emph{reduced} if the differential $\partial$ of $CFK^{\infty}(K)$ strictly drops either the Alexander filtration or $(-U)$-exponents filtration. It is known that every filtered chain complex is filtered chain homotopic to a reduced complex. 

\begin{defn}
A basis $\{\mathbf{x}_i\}$ for a filtered chain complex $(C, \partial)$ is called a \emph{filtered basis} if the set $\{ \mathbf{x}_i \ | \mathbf{x}_i \in C_S \}$ is a basis for a filtered subcomplex $C_S \subset C$. 
\end{defn}

Now we are able to define two different simplified bases for $CFK^{\infty}(K)$. 
\begin{defn}
A filtered basis $\{ \mathbf{x}_i \}$ over $\mathbb{F}_2 [U]$ for the reduced complex $CFK^{\infty}(K)$ is \emph{vertically simplified} if for each basis element $\mathbf{x}_i$, exactly one of the following holds:
\begin{itemize}
  \item $\mathbf{x}_i$ is in the image of $\partial^v$ and there exists a unique basis element $\mathbf{x}_{i-1}$ such that $\partial^v \mathbf{x}_{i-1} = \mathbf{x}_i$.
  \item $\mathbf{x}_i$ is in the kernel, but not in the image of $\partial^v$.
  \item $\mathbf{x}_i$ is not in the kernel, and $\partial^v \mathbf{x}_i = \mathbf{x}_{i+1}$.
\end{itemize}  
\end{defn}
If $\partial^v \mathbf{x}_i = \mathbf{x}_{i+1}$, then we say that there is a \emph{vertical arrow} from $\mathbf{x}_i$ to $\mathbf{x}_{i+1}$ and the \emph{length of the arrow} is $A(\mathbf{x}_i)- A( \mathbf{x}_{i+1})$. Since $H_* (C^v) \cong \mathbb{F}_2$, there is a distinguished element that generates the homology. By reordering, we let $\mathbf{x}_0$ denote the element. \\  

Let $\{ \mathbf{x}_i \}$ be a filtered basis for the reduced complex $CFK^{\infty}(K)$. Note that the set of elements $\{ U^{m_i} \cdot \mathbf{x}_i \}$, where $m_i := A( \mathbf{x}_i)$, induces a basis for $C^h$. 
\begin{defn}
A filtered basis $\{ \mathbf{x}_i \}$ over $\mathbb{F}_2 [U]$ for the reduced complex $CFK^{\infty}(K)$ is \emph{horizontally simplified} if each $\mathbf{x}_i$ satisfies exactly one of the following:
\begin{itemize}
  \item $U^{m_i} \mathbf{x}_i$ is in the image of $\partial^h$ and there exists a unique basis element $\mathbf{x}_{i-1}$ such that $\partial^h U^{m_{i-1}} \mathbf{x}_{i-1} = U^{m_i} \mathbf{x}_i$.
  \item $U^{m_i} \mathbf{x}_i$ is in the kernel, but not in the image of $\partial^h$.
  \item $U^{m_i} \mathbf{x}_i$ is not in the kernel of $\partial^h$, and $\partial^h U^{m_i} \mathbf{x}_i = U^{m_{i+1}} \mathbf{x}_{i+1}$.
\end{itemize}
\end{defn}
Again, if $\partial^h U^{m_i} \mathbf{x}_i = U^{m_{i+1}} \mathbf{x}_{i+1}$, we say there is a \emph{horizontal arrow} from $\mathbf{x}_i$ to $\mathbf{x}_{i+1}$, and the \emph{length of the arrow} is $A(\mathbf{x}_i) - A( \mathbf{x}_{i+1})$. Also we let $\mathbf{x}_0$ denote the element that generates the homology $H_* (C^h)$. 

\subsection{Bordered Heegaard diagram of the knot complement}
For a doubly pointed Heegaard diagram $\mathcal{H}_K = ( \Sigma_0, \boldsymbol{\alpha}_0, \boldsymbol{\beta}_0, z, w)$ of a knot $K$, the bordered Heegaard diagram $\mathcal{H} (n)$ of the knot $K$ complement is obtained by attaching a two-dimensional one-handle to $\Sigma_0$ with one foot close to $z$ and the other foot close to $w$. Let $\Sigma$ denote the resulting genus $g$ surface and $m$ be a meridional circle on $\Sigma$. Then introduce a circle $\alpha_g$ in the two-dimensional one-handle parallel to the meridian $m$ of $K$, and a circle $\beta_g$ which transversely intersects $\alpha_g$ once and disjoint from other $\beta$ circles. \\

Let $\lambda$ be a circle in $\Sigma$, being a circle in $\Sigma$ a zero-framed longitude with respect to the Seifert framing and intersecting $\alpha_g$ once transversely and disjoint from other $\alpha$ circles. Now for a tubular neighborhood $\mathcal{W}$ of $\alpha_g$ in the two-dimensional one-handle, apply the Dehn twist on $\lambda$ along $\alpha_g$ inside $\mathcal{W}$ $n$-times, as in Figure~\ref{fig:windingregion}. Finally, puncture the intersection of $\lambda$ and $\alpha_g$, and label the four regions around the puncture 0, 1, 2, and 3 in a counterclockwise direction so that the $\alpha_g$ is dividing the regions 0, 1 and the regions 2, 3. On the punctured $\Sigma$, we let $\alpha^a_1 := \lambda$ and $\alpha^a_2 := \alpha_g$. Thus we define the bordered Heegaard diagram of the knot $K$ complement to be
\begin{displaymath}
\mathcal{H}(n) := ( \Sigma, \{ \alpha^a_1, \alpha^a_2 \} \cup \boldsymbol{\alpha}_0, \{ \beta_g \} \cup \boldsymbol{\beta}_0 ) 
\end{displaymath}

\begin{figure}
\begin{center}
\includegraphics{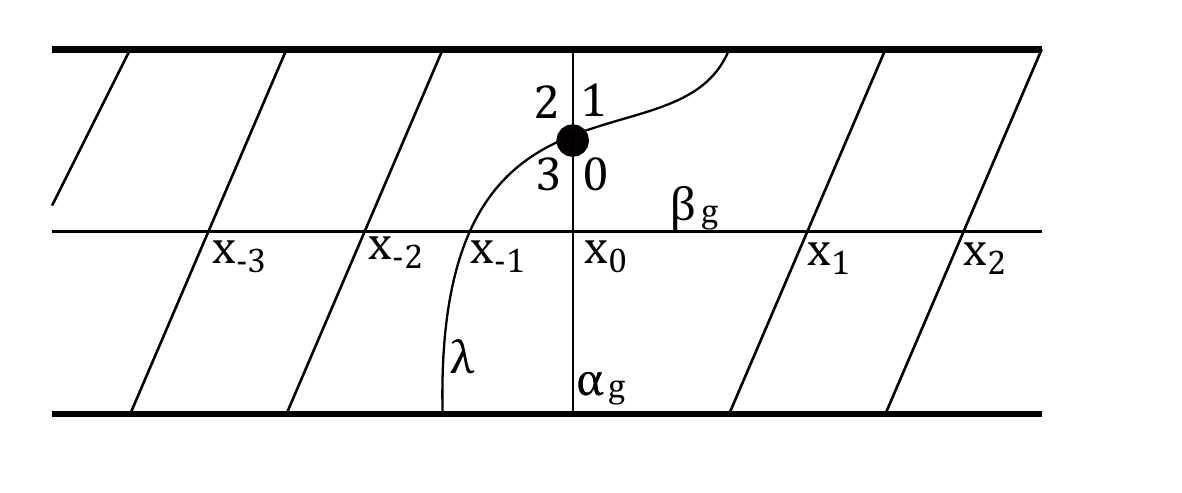}
\caption{The figure describes the winding region $\mathcal{W}$. The top bold line is identified to the bottom bold line so that it forms the two-dimensional one-handle. Our convention is that the left-end of $\mathcal{W}$ is attached to the region of $\Sigma_0$ near $w$ and the right-end is attached near $z$. The basepoint of the bordered Heegaard diagram is put on the region labelled 0.}
\label{fig:windingregion}
\end{center}
\end{figure}

Inside $\mathcal{W}$, the intersection point of $\alpha^a_2$ and $\beta_g$ will be called $x_0$. Then we label the intersection points of $\alpha^2_1$ and $\beta_g$ as follows. As travelling along the arc $\alpha^a_1$ in $\mathcal{W}$ between the region 2 and 3, the intersections points of $\alpha^2_1$ and $\beta_g$ will be labelled $x_{-1}, x_{-2}, \cdots$ in order. Likewise, the intersection points on the opposite side of $\alpha^a_2$ will be labelled $x_{+1}, x_{+2}, \cdots$ and so on. For simplicity, from now on we assume there are $n+1$ intersection points in $\mathcal{W}$ labelled as $x_{-\frac{n}{2}}, \cdots, x_{+\frac{n}{2}}$. Recall that $\mathfrak{S}_K$ denote the set of generators of $\widehat{CFK}(\mathcal{H}_K)$. We let $\mathfrak{S}(n)$ denote the set of generators of $\widehat{CFD} ( \mathcal{H}(n) )$. Then for each generator $\mathbf{x} \in \mathfrak{S}_K$, there are $n+1$ generators $\{ \mathbf{x}_i \}$ in $\mathfrak{S}(n)$, where the generator $\mathbf{x}_i$ is obtained by adding the point $x_i \in \mathcal{W}$ to $\mathbf{x}$. Note that for sufficiently large $n$, most generators in $\mathfrak{S}(n)$ can be written $\mathbf{x}_i$ for some $\mathbf{x} \in \mathfrak{S}_K$. The generators in $\mathfrak{S}(n)$ that cannot be written in the form of $\mathbf{x}_i$ for some $\mathbf{x} \in \mathfrak{S}_K$ are called \emph{exterior generators}. \\

\cite[Lemma 11.41]{LOT08} and \cite[Lemma 11.43]{LOT08} assert that there is a function $S : \mathfrak{S}(n) \rightarrow \frac{1}{2} \mathbb{Z}$ satisfying the following:
\begin{itemize}
  \item Let $A : \mathfrak{S}_K \rightarrow \mathbb{Z}$ be the (normalized) Alexander grading of elements in $\mathfrak{S}_K$. Then
  \begin{displaymath}
    S( \mathbf{x}_k ) = A (\mathbf{x} ) - k + \left( 
    \frac{(n+1) \cdot \mathrm{sgn}(k) }{2} \right),
  \end{displaymath}
  where $\textrm{sgn}(k) = -1, 0$, or 1 if $k<0, k=0$ or $k>0$ respectively. In particular, $S( \mathbf{x}_0 ) = A (\mathbf{x})$. 
  \item There exists a constant $c$ satisfying the following: let $\mathbf{y} \in \mathfrak{S}(n)$ with $| S (\mathbf{y}) | \geq c$. Then $\mathbf{y}$ is not an exterior generator, i.e, there exists $\mathbf{x} \in \mathfrak{S}_K$ such that $\mathbf{y} = \mathbf{x}_i$ for some $i \in \mathbb{Z}$. Moreover, the sign of $S (\mathbf{y})$ agrees with the sign of $i$.
\end{itemize} 

\subsection{Coefficient maps and their domains}

Let $n$ be a sufficiently large integer. \cite[Theorem 11.36]{LOT08} gives the algorithm to compute the homotopy type of $\widehat{CFD}( \mathcal{H}(n))$ of knot complement with framing $-n$ from the knot Floer complex $CFK^- (K)$. Let $\{ \xi_i \}$ be a vertically simplified basis. Then we may regard $\{ \xi_i \}$ to be a basis of $\imath_1 \widehat{CFD} (\mathcal{H} (n))$. For each arrow of length $l$ from $\xi_i$ to $\xi_{i+1}$, then we have basis elements $\kappa^i_1, \cdots, \kappa^i_k \in \imath_2 \widehat{CFD} (\mathcal{H}(n))$ that form the following sequence:
\begin{displaymath}
\xymatrix{
\xi_i \ar[r]^{\rho_1} & \kappa^i_1 & \cdots \ar[l]_{\rho_{23}} & \kappa^i_k \ar[l]_{\rho_{23}} & \kappa^i_{k+1} \ar[l]_{\rho_{23}} & \cdots \ar[l]_{\rho_{23}} & \kappa^i_l \ar[l]_{\rho_{23}} & \xi_{i+1}. \ar[l]_{\rho_{123}}
}
\end{displaymath}  

Similarly, let $\{ \eta_i \}$ be a set of horizontally simplified basis for $CFK^-(K)$. For an arrow of length $l$ from $\eta_i$ to $\eta_{i+1}$, we also have basis elements $\lambda^i_1, \cdots, \lambda^i_k \in \imath_2 \widehat{CFD} (\mathcal{H}(n))$. Again these form the following sequence:
\begin{displaymath}
\xymatrix{
\eta_i \ar[r]^{\rho_3} & \lambda^i_1 \ar[r]^{\rho_{23}} & \cdots \ar[r]^{\rho_{23}} & \lambda^i_k \ar[r]^{\rho_{23}} & \lambda^i_{k+1} \ar[r]^{\rho_{23}} & \cdots \ar[r]^{\rho_{23}} & \lambda^i_l \ar[r]^{\rho_2} & \eta_{i+1}.
}
\end{displaymath}

Recall that there are two distinguished basis elements $\xi_0$ and $\eta_0$, which generate homologies $H_* (C^v)$ and $H_* (C^h)$ respectively. The \emph{unstable chain} is a string of elements $\mu_1, \cdots, \mu_m \in \imath_2 \widehat{CFD} (\mathcal{H}(n))$ connecting $\xi_0$ and $\eta_0$ as follows:
\begin{displaymath}
\xymatrix{
\xi_0 \ar[r]^{\rho_1} & \mu_1 & \mu_2 \ar[l]_{\rho_{23}} & \cdots \ar[l]_{\rho_{23}} & \mu_m \ar[l]_{\rho_{23}} & \eta_0, \ar[l]_{\rho_3}
}
\end{displaymath}  
where $m = n+ 2\tau(K)$. \\

Moreover, if the framing equals two times of the smooth concordance invariant $\tau(K)$, then the unstable chain reduces to
\begin{displaymath}
\xymatrix{
\xi_0 \ar[r]^{\rho_{12}} & \eta_0.
}
\end{displaymath}

\begin{figure}
\begin{center}
\includegraphics[scale=0.6]{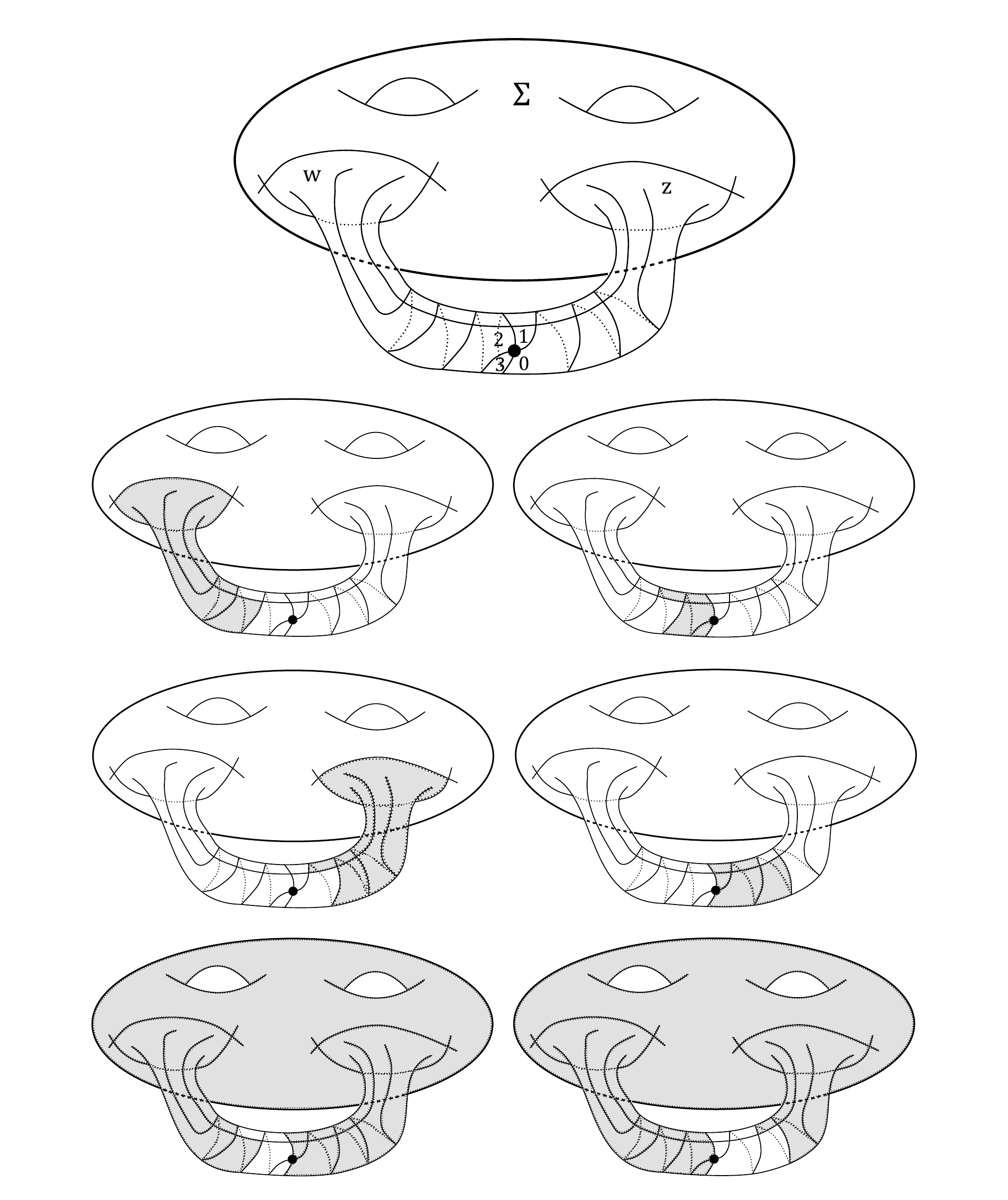}
\caption{The top figure represents the bordered Heegaard surface $\Sigma$ after attaching the winding region. Six diagrams with shaded regions depict examples of domains used in~\cite[Theorem 11.36]{LOT08}. These figures, beginning from the top right figure and in clockwise direction, represent type 1, type 2, type 3, type 4, type 5 and type 6 domains.}
\label{fig:domains}
\end{center}
\end{figure}

In the rest of this subsection, we will focus on the domains of the diagram $\mathcal{H}(n)$ that contributes to the differentials $\delta^1$ of $\widehat{CFD} (\mathcal{H}(n))$. \\

The \emph{coefficient map} $D_I : \mathfrak{S}(n) \rightarrow  \mathfrak{S}(n)$, indexed by an increasing sequence of consecutive integers $I = \{ i_0, \cdots, i_n \} \subset \{ 1,2,3 \}$ (including the empty sequence $\emptyset$), satisfies the following relation
\begin{displaymath}
\delta^1 = 1 \otimes D_{\emptyset} + \sum_i \rho_i \otimes D_i + \sum_{ \{ i,j | j=i+1 \} } \rho_{ij} \otimes D_{ij} + \rho_{123} \otimes D_{123}.
\end{displaymath}
The existence of a coefficient map $D_I$ implies there is a domain in $\mathcal{H}(n)$ between the associated elements adjacent to the boundary of $\mathcal{H}(n)$ labelled as $I$. In general, $I$ can be any interval in $\{ 0, 1, 2, 3 $\} with respect to the cyclic ordering, including the empty interval; e.g., $I=01$ or $30$ are possible. However, any coefficient map of interval including 0 will not have contribution towards the differential $\delta^1$. \\

First, recall that for any generator $\mathbf{x} \in \mathfrak{S}_K$ corresponds to $\mathbf{x}_k \in \widehat{CFD} (\mathcal{H}(n))$, $k = -\frac{n}{2}, \cdots, \frac{n}{2}$. Then the domains between these elements, whose modulo two count of index one moduli space equals one, can be described as below.  
\begin{displaymath}
\xymatrix{
\cdots & \mathbf{x}_{-3} \ar[l]_{D_{23}} & \mathbf{x}_{-2} \ar[l]_{D_{23}} & \mathbf{x}_{-1} \ar[l]_{D_{23}} & \mathbf{x}_0 \ar[l]_{D_3} \ar[r]^{D_1}& \mathbf{x}_1 \ar[r]^{D_{01}} & \mathbf{x}_2 \ar[r]^{D_{01}} & \mathbf{x}_3 \ar[r]^{D_{01}} & \cdots
}
\end{displaymath}
From the above sequence, the domains associated to the coefficient map $D_I$, $I=3$ and $23$, will be called \emph{type 1 domains} in this paper. The homology class of the domain of the coefficient map $D_{23} : \mathbf{x}_{-i} \rightarrow \mathbf{x}_{-i-1}$, $i \geq 1$, will be written as $\phi_i \in \pi_2 ( \mathbf{x}_{-i}, \mathbf{x}_{-i-1} )$, and the homology class of the domain of the coefficient map $D_3 : \mathbf{x}_0 \rightarrow \mathbf{x}_{-1}$ will also be written as $\phi_0 \in \pi_2 ( \mathbf{x}_0, \mathbf{x}_{-1} )$. Likewise, the homology classes of the domains of the coefficient map $D_1 : \mathbf{x}_0 \rightarrow \mathbf{x}_1$ and $D_{01} : \mathbf{x}_i \rightarrow \mathbf{x}_{i+1}$ will be written as $\overline{\phi}_0$ and $\overline{\phi}_i$, respectively. In this paper, these two kinds of domains are called \emph{type 2 domains}. See Figure~\ref{fig:domains}. \\

There are coefficient maps $D_{23} : \mathbf{x}_i \rightarrow \mathbf{x}_{i-1}$, $i \geq 2$. The domain of this map can be understood as the entire surface $\Sigma$ minus the domain of the homology class $\overline{\phi}_i$. The homology class of this coefficient map will be denoted by $\psi_i$. By taking advantage of the symmetry of the diagram, there are coefficient maps $D_{01} : \mathbf{x}_{-i} \rightarrow \mathbf{x}_{-i+1}$, $i \geq 2$. Again the homology classes associated to these coefficient maps are denoted by $\overline{\psi}_i$, whose domains can be graphically depicted as $\Sigma$ minus the domain of the homology class $\phi_{i-1}$. We will call the domains of the homology classes $\psi_i$ and $\overline{\psi}_i$ \emph{type 3 domains} and \emph{type 4 domains}, respectively. \\

Lastly, there is a coefficient map $D_I$ of the empty interval $I =\emptyset$. These maps appear if there is a vertical or horizontal arrow between simplified bases. Suppose there is a vertical arrow $\mathbf{x} \rightarrow \mathbf{y}$ of length $l$. Then we have a coefficient map $D_{\emptyset} : \mathbf{x}_i \rightarrow \mathbf{y}_{i-l}$, $i > l$. To describe the domain, we need to consider the winding region $\mathcal{W}$ detached from $\Sigma$. Then it is $S^2$ minus two punctures, with each puncture previously attached to the points $z$ or $w$. Cutting open the winding region along the circle $\alpha_g$, we obtain two annuli. For the annulus that contains the region labelled as 0 and 1, the interior of the annulus has a part of $\beta_g$ and $\lambda = \alpha^a_1$, and its outer boundary is identified to $\alpha_g$. Then the interior has intersection points $x_1, x_2, \cdots$, and between them we have a homology class $\zeta_i \in \widetilde{\pi}_2 (x_i, x_{i-1})$ (thought as a homology class in the winding region, see Figure~\ref{fig:hom_class_wind}). Consider the homology class $\zeta_i * \zeta_{i-1} * \cdots * \zeta_{i-l+1}$. The domain of this homology class has multiplicity $l$ near the puncture. The gluing of this domain with the domain of the homology class between $\mathbf{x}$ and $\mathbf{y}$ as in $CFK^-(K)$, which has multiplicity $l$ near $z$ as well, gives the domain of the homology class $\varphi_i \in \pi_2 (\mathbf{x}_i, \mathbf{y}_{i-l})$, $i>l$. (If $i=l$, there is a coefficient map $D_0 : \mathbf{x}_l \rightarrow \mathbf{y}_0$.) The domains of the homology classes $\varphi_i$, $i \geq l$ will be called \emph{type 5 domains}. Likewise, the symmetry of the diagram will give coefficient maps $D_{\emptyset} : \mathbf{x}_{-i} \rightarrow \mathbf{y}_{-i+l}$, $i >l$ for a horizontal arrow $\mathbf{x} \rightarrow \mathbf{y}$ of length $l$. (Again, there is a coefficient map $D_2 : \mathbf{x}_{-l} \rightarrow \mathbf{y}_0$ if $i=l$.) The homology classes of these coefficient maps are written $\overline{\varphi}_i \in \pi_2 (\mathbf{x}_{-i} , \mathbf{y}_{-i+l})$, $i \geq l$. The domains of $\overline{\varphi}_i$ will be called \emph{type 6 domains}. 

\begin{figure}
\begin{center}
\includegraphics[scale=0.7]{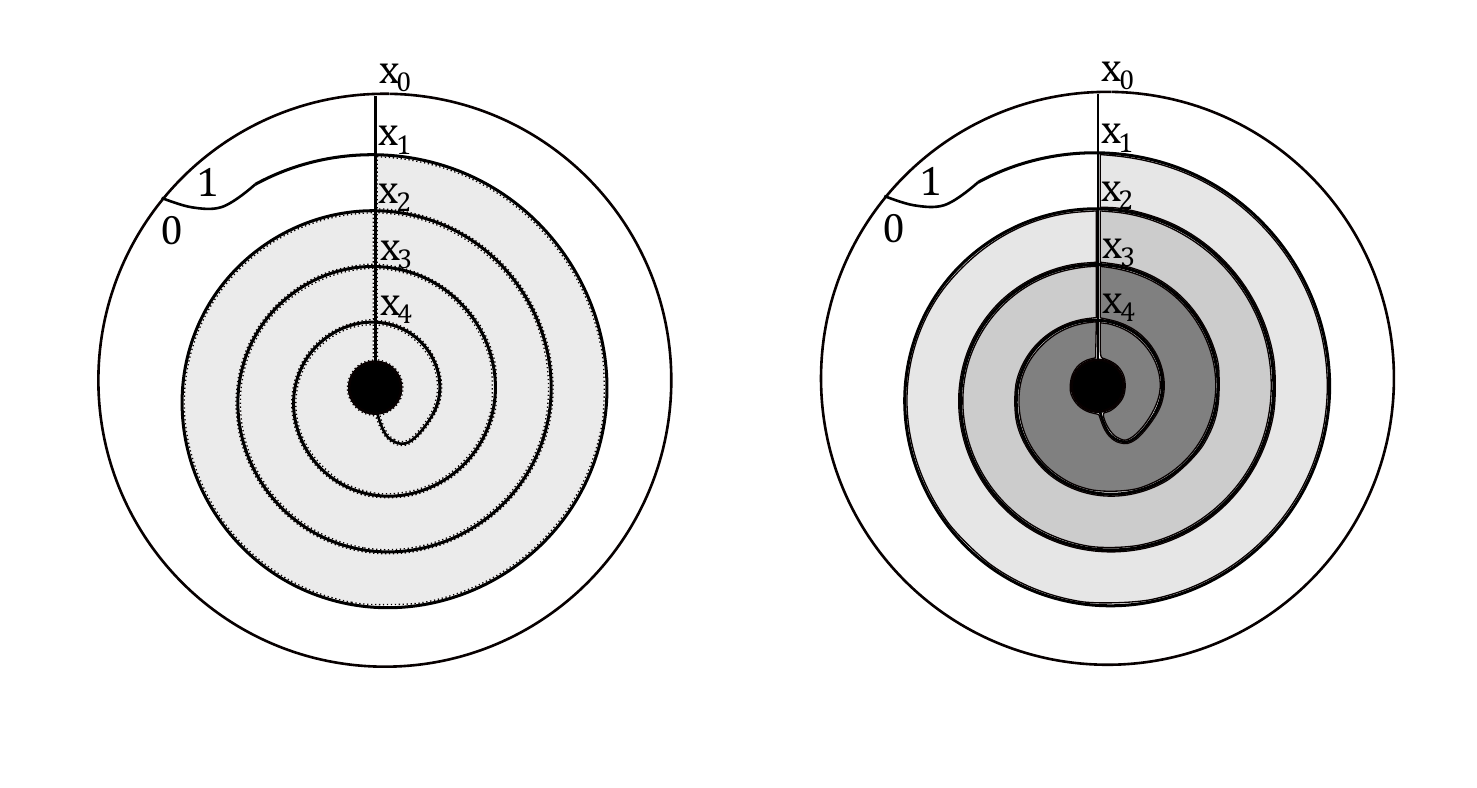}
\caption{We illustrate the domains of homology class $\zeta_2 \in \widetilde{\pi_2} (x_2, x_1)$ on the right, and the domain of $\zeta_4 * \zeta_3 * \zeta_2$ on the left. Note that the shading of each region is darker as the multiplicity of the region increases. The middle black dot in each diagram should be attached near point $z \in \Sigma_0$.}
\label{fig:hom_class_wind}
\end{center}
\end{figure}

\begin{rem}
The coefficient map of the interval 123 only appears whenever there is a horizontal arrow $\mathbf{x} \rightarrow \mathbf{y}$ of length one. In this case, the map is
\begin{displaymath}
D_{123} : \mathbf{x}_0 \rightarrow \mathbf{y}_1.
\end{displaymath}
\end{rem}

\cite[Figure 11.13]{LOT08} illustrates the type-$D$ module $\widehat{CFD} (\mathcal{H}(n))$ obtained by the domains considered above. The complex is quite involved, but collapsing coefficient maps $D_{\emptyset}$ will result in the simplified complex as stated in the beginning of this subsection.  

\subsection{Gradings}
The bordered Floer package of torus algebra is also endowed with relative grading group $G$. Here we describe the grading group of a single boundary case. Let $G$ be generated by triples $(j;p,q)$, $j,p,q \in \frac{1}{2} \mathbb{Z}$. The multiplication is defined as
\begin{displaymath}
(j_1; p_1, q_1) \cdot (j_2; p_2, q_2) = \left( j_1 + j_2 + 
\left| 
  \begin{array}{cc}
  p_1 & q_1 \\
  p_2 & q_2 
  \end{array}
\right| ; 
p_1 + p_2, q_1 + q_2 \right).
\end{displaymath}
The distinguished central element $(1;0,0)$ is written as $\lambda$. \\

For each element in the torus algebra $\mathcal{A}$, the grading is given by the following rule. For $\rho_i$, $i=1,2,3$, we define the grading $\mathrm{gr}(\rho_i) \in G$ as follows.
\begin{eqnarray*}
\mathrm{gr}(\rho_1) & = & \left( - \frac{1}{2} ; \frac{1}{2}, - \frac{1}{2} \right) \\
\mathrm{gr}(\rho_2) & = & \left( - \frac{1}{2} ; \frac{1}{2}, \frac{1}{2} \right) \\
\mathrm{gr}(\rho_3) & = & \left( - \frac{1}{2} ; - \frac{1}{2}, \frac{1}{2} \right) ; \\
\end{eqnarray*}
and we let the grading respects the multiplication so that $\mathrm{gr}(\rho_I \cdot \rho_J) = \mathrm{gr}(\rho_I) \cdot \mathrm{gr}(\rho_J)$. \\

For a Heegaard diagram $\mathcal{H}$, suppose that there is a homology class $B \in \pi_2 (\mathbf{x}, \mathbf{y})$, possibly adjacent to the boundary $\partial \mathcal{H}$. If so, the boundary adjacency can be written as  $c_1 \cdot \rho_1 + c_2 \cdot \rho_2 + c_3 \cdot \rho_3$. Then the grading $\mathrm{gr}(B)$ of $B$ is defined as
\begin{displaymath}
\mathrm{gr}(B) := ( -e(B) -n_{\mathbf{x}}(B) -n_{\mathbf{y}}(B) ; \frac{c_1 + c_2 -c_3}{2}, \frac{-c_1 + c_2 + c_3}{2} ).
\end{displaymath}
Then we define $\mathrm{gr}(\mathbf{y}) := \mathrm{gr}(\mathbf{x}) \mathrm{gr} (B)$. Of course this grading is not well-defined, but we can remove the uncertainty of the grading by taking the grading set not on $G$, but on $G / P$, where $P := \{ \mathrm{gr}(B) | B \in \pi_2 (\mathbf{x}, \mathbf{x}) \}$ is the subgroup of $G$. Under this setting, the grading of type-$D$ module $\widehat{CFD}(\mathcal{H})$ is known to be
\begin{displaymath}
\mathrm{gr}( \partial \mathbf{x} ) = \lambda^{-1} \mathrm{gr} (\mathbf{x}).
\end{displaymath}
In terms of coefficient map, $D_I : \mathbf{x} \rightarrow \mathbf{y}$ enables to track the grading difference between $\mathbf{x}$ and $\mathbf{y}$; specifically, 
\begin{displaymath}
\lambda^{-1} \mathrm{gr} (\mathbf{x}) = \mathrm{gr}(\partial \mathbf{x}) = \mathrm{gr} (\rho_I \mathbf{y}) = \mathrm{gr}(\rho_I) \mathrm{gr} (\mathbf{y}) 
\end{displaymath}
and this leads to
\begin{displaymath}
\mathrm{gr}( \mathbf{y} ) = \lambda^{-1} \mathrm{gr}(\rho_I)^{-1} \mathrm{gr}(\mathbf{x}). 
\end{displaymath}
For $(j;a,b) \in G$, $j$ is called the \emph{Maslov component} and $(a,b)$ is called the \emph{$spin^c$-component}. \\

Lastly, the grading set of $\widehat{CFA}(\mathcal{H}_1) \boxtimes \widehat{CFD}(\mathcal{H}_2)$ is the double quotient group $P_1 \backslash G / P_2$, where $P_i$ is the subgroup of $G$ generated by gradings of periodic domains in $\mathcal{H}_i$.

\section{Heegaard diagram of connected sum}
\label{sec:diagram}

Let $L_1$ and $L_2$ be two components of Hopf link $\mathcal{L} \subset S^3$. For a knot $K$ in $S^3$, taking connected sum of $K$ and $L_1$ will result in a 2-link $\mathcal{L}_K$ consisting of the knot $K$ and its meridian $L_2$. \\   

The Heegaard diagram of a connected sum of two knots has been explained in~\cite[Section 7]{OZ04}, and we will draw the doubly bordered Heegaard diagram of $S^3 \backslash \nu (\mathcal{L}_K)$ by applying the same procedure. First note the doubly bordered Heegaard diagram of the Hopf link $\mathcal{L}$ complement in $S^3$ is the same as the diagram of the identity bimodule as in \cite[Figure 13]{LOT11}. See the top figure of Figure~\ref{fig:connectedsum}. The diagram will be written as $\Sigma_{\mathcal{L}}$. For convenience, the two punctures of $\Sigma_{\mathcal{L}}$ will be referred to the ``left'' and ``right'' punctures. $\Sigma_{\mathcal{L}}$ has four $\alpha$-curves, namely $\widetilde{\alpha}^{a,L}_1$, $\widetilde{\alpha}^{a,L}_2$, $\widetilde{\alpha}^{a,R}_1$, and $\widetilde{\alpha}^{a,L}_2$. It also has two $\beta$-circles called $\widetilde{\beta}_1$ and $\widetilde{\beta}_2$, and an arc $\widetilde{ \mathbf{z} }$ connecting two punctures as well. Regard $\widetilde{\alpha}^{a,L}_1$ and $\widetilde{\alpha}^{a,R}_1$ as the longitudinal curves of left and right knot components of $\mathcal{L}$; thus $\widetilde{\alpha}^{a,L}_2$ and $\widetilde{\alpha}^{a,R}_2$ are the meridional ones. \\

We modify the diagram to get the doubly bordered Heegaard diagram of $\mathcal{L}_K$ complement. First, make a hole $C$ on $\widetilde{\alpha}^{a,L}_1$ near the left puncture of the diagram $\Sigma_{\mathcal{L}}$. Then draw a curve $\widetilde{\alpha}_s$ whose ends are lying on $\partial C$, parallel to $\widetilde{\alpha}^{a,L}_2$. Then the diagram has five intersection points defined as
\begin{displaymath}
\mathbf{a} := \widetilde{\alpha}^{a,L}_2 \cap \widetilde{\beta}_1, \quad \mathbf{b} := \widetilde{\alpha}^{a,R}_1 \cap \widetilde{\beta}_1, \quad \mathbf{c} := \widetilde{\alpha}^{a,L}_1 \cap \widetilde{\beta}_2, \quad \mathbf{d} := \widetilde{\alpha}^{a,R}_2 \cap \widetilde{\beta}_2, \quad \mathbf{a_s} := \widetilde{\alpha}_s \cap \widetilde{\beta}_1. 
\end{displaymath}
Now, for a bordered Heegaard diagram $\{ \Sigma_K, \overline{\boldsymbol{\alpha}}_K, \boldsymbol{\beta}_K, z_K \}$ of a knot $K$ complement (with sufficiently large negative framing), we identify the puncture of $\Sigma_K$ to $C$. Precisely, the identification is made so that
\begin{itemize}
  \item the ends of meridional curve $m \in \overline{ \boldsymbol{\alpha} }_K$ is connected to the ends of $\widetilde{\alpha}_s$, and the ends of longitudinal curve $\lambda \in \overline{ \boldsymbol{\alpha} }_K$ to the ends of $\widetilde{\alpha}^{a,L}_1$;
  \item the region of $\Sigma_K$ labelled as 2 is glued to the region of $\Sigma_{\mathcal{L}}$ labelled as 2 adjacent to the left puncture; and  
  \item the region of $\Sigma_K$ labelled as 3 is glued to the region of $\Sigma_{\mathcal{L}}$ labelled as 3 adjacent to the left puncture.
\end{itemize}
Therefore, the region labelled as 1 is glued to the region of $\Sigma_{\mathcal{L}}$ labelled as 2 adjacent to the right puncture. The resulting diagram will be henceforth called $\Sigma$. Note that $\Sigma$ has two punctures, again called left and right punctures. \\

\begin{figure}
\begin{center}
\includegraphics[scale=0.8]{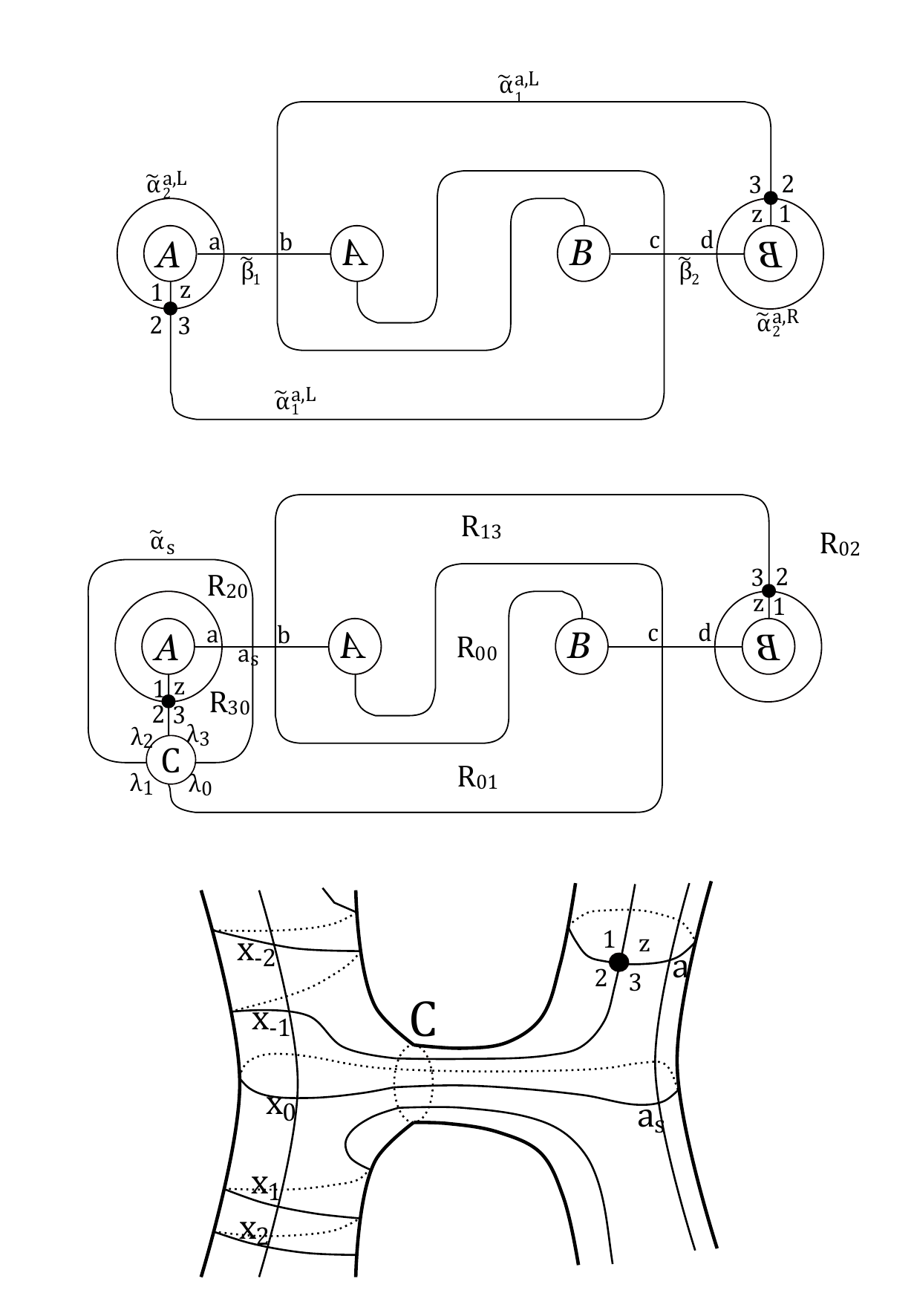}
\caption{The top figure is the doubly bordered Heegaard diagram $\Sigma_{\mathcal{L}}$ of Hopf link complement, by identifying two pairs of circles labelled $A$ and $B$. The labelling of curves and  intersection points are written as above. The two black dots represent the left and right boundaries of $\Sigma$. The middle figure shows the modified diagram after making a hole $C$ and drawing a curve $\widetilde{\alpha}_s$. The bottom figure shows the region near $C$ of the resulting diagram after the modification.}
\label{fig:connectedsum}
\end{center}
\end{figure}

Then the resulting diagram has the following data that gives the doubly bordered Heegaard diagram $\mathcal{H}_{\mathcal{L}_K} (n) := \{ \Sigma, \overline{\boldsymbol{\alpha}}, \boldsymbol{\beta}, \mathbf{z} \}$ of link $\mathcal{L}_K = K \coprod L_2$ complement, such that $K$ has framing $-n$ and the meridian $L_2$ has framing zero.
\begin{itemize}
  \item The surface $\overline{\Sigma}$ with two boundary components $\partial_L \overline{\Sigma}$ and $\partial_R \overline{\Sigma}$, which are the boundaries of left and right punctures, respectively.
  \item Let $\alpha^{a,L}_1$ be the curve obtained by connecting $\lambda$ and $\widetilde{\alpha}^{a,L}_1$, and $\alpha_s$ be a circle obtained by connecting $\widetilde{\alpha}_s$ and $m$. The set of $\alpha$-curves $\overline{\boldsymbol{\alpha}}$ is defined as
    \begin{displaymath}
    \overline{\boldsymbol{\alpha}} := \left(  \overline{\boldsymbol{\alpha}}_K \backslash \{ m, \lambda \} \right) \ \cup \ \{ \widetilde{\alpha}^{a,R}_1, \widetilde{\alpha}^{a,R}_2, \alpha^{a,L}_1, \widetilde{\alpha}^{a,L}_2 \}  \ \cup \ \{ \alpha_s \}.
    \end{displaymath}
    By rearranging indices, we let $\alpha^{a,R}_1 := \widetilde{\alpha}^{a,R}_1$, $\alpha^{a,R}_2 := \widetilde{\alpha}^{a,R}_2$, and $\alpha^{a,L}_2 := \widetilde{\alpha}^{a,L}_2$.
  \item The set of $\beta$-circles $\boldsymbol{\beta}$ carries over; i.e., $\boldsymbol{\beta} := \boldsymbol{\beta}_K \cup \{ \widetilde{\beta}_1, \widetilde{\beta}_2 \}$. Again, we write $\beta_1 := \widetilde{\beta}_1$ and $\beta_2 := \widetilde{\beta}_2$ for notational simplicity.
  \item The arc $\mathbf{z}$ is the same as $\widetilde{\mathbf{z}}$.
\end{itemize}
This diagram has two boundary components; the left (respectively, right) boundary algebra is labelled $\rho_I$ (respectively $\sigma_I$), where $I = \{ 1, 2, 3, 12, 23, 123 \}$. \\

A reader may observe that for $\mathbf{x}_k \in \mathfrak{S}(n)$, we can associate generators of $\widehat{CFDD}( \mathcal{H}_{\mathcal{L}_K} (n) )$ to $\mathbf{x}_k$, as described below. 
\begin{itemize}
  \item If $k \neq 0$, $\mathbf{x}_k \mathbf{a_s d}$ is a tuple of intersection points between $\overline{\boldsymbol{\alpha}}$ and $\boldsymbol{\beta}$ consisting of intersection points of $\mathbf{x}_k$ and $\mathbf{a_s, d}$. These generators belong to $\imath_2 \jmath_1 \widehat{CFDD} ( \mathcal{H}_{\mathcal{L}_K} (n) )$.
  \item If $k=0$, $\mathbf{x}_0 \mathbf{a d}$ and $\mathbf{x}_0 \otimes \mathbf{b c}$ are tuples of intersection points consisting of $\mathbf{x}_0$ and $\mathbf{a, d}$ (respectively, $\mathbf{b, c}$). The generators $\mathbf{x}_0 \mathbf{ad} \in \imath_1 \jmath_1 \widehat{CFDD} ( \mathcal{H}_{\mathcal{L}_K} (n) )$ and $\mathbf{x}_0 \mathbf{bc} \in \imath_2 \jmath_2 \widehat{CFDD} ( \mathcal{H}_{\mathcal{L}_K} (n) )$.
\end{itemize}  
We have $\mathbb{F}_2$-vector space isomorphisms $\imath_1 \jmath_1 M \cong CFK^-(K)$ and $\imath_2 \jmath_2 M \cong CFK^-(K)$.

\section{Holomorphic curves of domains obtained by gluing}
\label{sec:main}

In this section, we compute the differential $\delta^1 :M \rightarrow \mathcal{A}_L \otimes \mathcal{A}_R \otimes M$ of the type-$DD$ module $M : = \widehat{CFDD}(\mathcal{H}_{\mathcal{L}_K} (n) )$. \\

The strategy of the computation is as follows: The bordered Heegaard surface $\Sigma$ is basically obtained by gluing two surfaces $\Sigma_K$ and $\Sigma_{\mathcal{L}}$. Every non-provincial domain of $\Sigma_K$ must be are glued to appropriate regions of $\Sigma_{\mathcal{L}}$ minus the hole $C$. Note that there are six of such regions in $\Sigma_{\mathcal{L}} \backslash C$, and we will name them as below (see the middle of Figure~\ref{fig:connectedsum}):
\begin{itemize}
  \item $R_{20}$, the region adjacent to the left boundary labelled 2;
  \item $R_{30}$, the region adjacent to the left boundary labelled 3;
  \item $R_{13}$, the region adjacent to the left boundary labelled 1 and the right boundary labelled 3;
  \item $R_{02}$, the region adjacent to the right boundary labelled 2;
  \item $R_{01}$, the region adjacent to the right boundary labelled 1;
  \item $R_{00}$, the region contains the arc $\mathbf{z}$;
\end{itemize}
As explained above, all non-provincial domains of $\Sigma_K$ can be classified into six types, and the domains of each type will be glued to one (or more) of the above regions listed. The resulting domains will possibly contribute to the differential $\delta^1$, as long as the expected dimension equals one. \\ 

{\bf Type 1 domain}. We investigate domains obtained by gluing type 1 domain of $\Sigma_K$ to region in $\Sigma_{\mathcal{L}} \backslash C$ with Maslov index one. Let $\mathbf{x} \in \mathfrak{S}_K$. Then there are following domains to consider:
\begin{itemize}
  \item The domain of $\phi_i \in \pi_2 (\mathbf{x}_{-i}, \mathbf{x}_{-i-1})$, $i \geq 1$, glued to $R_{20}+R_{30}$;
  \item The domain of $\phi_0 \in \pi_2 (\mathbf{x}_0, \mathbf{x}_{-1})$ glued to $R_{30}$.
\end{itemize}

The second domain listed above is rectangular, thus the holomorphic representative of $\phi_0 \in \pi_2 (\mathbf{x}_0, \mathbf{x}_{-1})$ results $\# \mathcal{M} (\mathbf{x}_0 \mathbf{ad}, \mathbf{x}_{-1} \mathbf{a_s d} ; \rho_3) = 1$. So it remains to investigate the first domain.

\begin{lem}
For $\mathbf{x} \in \mathfrak{S}_K$, the moduli space $\mathcal{M} ( \mathbf{x}_{-i} \mathbf{a_s d}, \mathbf{x}_{-i-1} \mathbf{a_s d} ; (\rho_2, \rho_3) )$, $i \geq 1$, has modulo two count one.
\label{lem:dilation1}
\end{lem}
\begin{proof}
We will apply the pairing theorem and use notations introduced in \cite[Chapter 9]{LOT08}. Let us consider a domain of $\phi_i \in \pi_2 (\mathbf{x}_{-i}, \mathbf{x}_{-i-1})$, $i \geq 1$, glued to $R_{20} + R_{30}$. In this proof, this domain will be called $D$. The domain $D$ is divided by $\partial C$. We then have two Reeb chords on $\partial C \cap D$, which will be labelled $\lambda_2$ and $\lambda_3$ so that $\lambda_j$ is lying on the region adjacent to $\rho_i$, $j=1,2$. See Figure~\ref{fig:dilation1} for the illustration. Recall that for the holomorphic representative of $\phi_i \in \pi_2 (\mathbf{x}_{-i}, \mathbf{x}_{-i-1})$, $i \geq 1$, the only valid  interpretation was $\mathcal{M} ( \mathbf{x}_{-i}, \mathbf{x}_{-i-1} ; (\lambda_2, \lambda_3) )$, which is basically a bigon \cite[Lemma 11.46]{LOT08}. This implies $\mathcal{M}( \mathbf{x}_{-1}, \mathbf{x}_{-i-1} ; ( \lambda_2, \lambda_3) )$ has a single curve $u$. Then the height difference between these two chords is written as
\begin{displaymath}
\mathrm{ev}_{\lambda_2, \lambda_3}= \mathrm{ev}_{\lambda_2} (u) - \mathrm{ev}_{\lambda_3} (u),
\end{displaymath}
which is a fixed real number, say $t_0$.  \\
On the other hand, let us consider the domains $R_{20}$ and $R_{30}$. The relevant moduli space $\mathcal{M} ( \mathbf{a_s}, \mathbf{a_s} ; R_{20}+ R_{30} )$, in order for the pairing theorem, has a pair of curves $v_1$ and $v_2$, where $v_1$ is associated to the moduli space $\mathcal{M} ( \mathbf{a_s}, \mathbf{a} ; R_{20} )$ and where $v_2$ is to $\mathcal{M} ( \mathbf{a}, \mathbf{a_s} ; R_{30} )$. Then the moduli space of $T$-matched pairs is
\begin{displaymath}
\widetilde{\mathcal{MM}}( T; \mathbf{x}_{-i}, \mathbf{x}_{-i-1} ; \mathbf{a_s}, \mathbf{a_s}) = \widetilde{\mathcal{M}} (\mathbf{x}_{-i}, \mathbf{x}_{-i-1}) \times_{T \cdot ev_1 = ev_2} \widetilde{\mathcal{M}} (\mathbf{a_s}, \mathbf{a_s} ).
\end{displaymath}
Then the moduli space 
\begin{displaymath}
\mathcal{MM} ( T; \mathbf{x}_{-i}, \mathbf{x}_{-i-1} ; \mathbf{a_s}, \mathbf{a_s}) : = \widetilde{\mathcal{MM}}( T; \mathbf{x}_{-i}, \mathbf{x}_{-i-1} ; \mathbf{a_s}, \mathbf{a_s}) / \mathbb{R}
\end{displaymath}
consists of the curve $u$ and a pair of disks from $R_{20}$ and $R_{30}$ with height difference $T \cdot t_0$. The modulo two count of this moduli space is unchanged for the sufficiently large $T$, thus letting $T \rightarrow \infty$ the height difference is going to $\infty$, so it proves the moduli space of the fibered product 
\begin{displaymath}
\mathcal{M} ( \mathbf{x}_{-i} \mathbf{a_s}, \mathbf{x}_{-i-1} \mathbf{a_s} ; (\rho_2, \rho_3)) 
\end{displaymath}
has modulo two count one, as desired.
\end{proof}

\begin{figure}
\begin{center}
\includegraphics[scale=0.6]{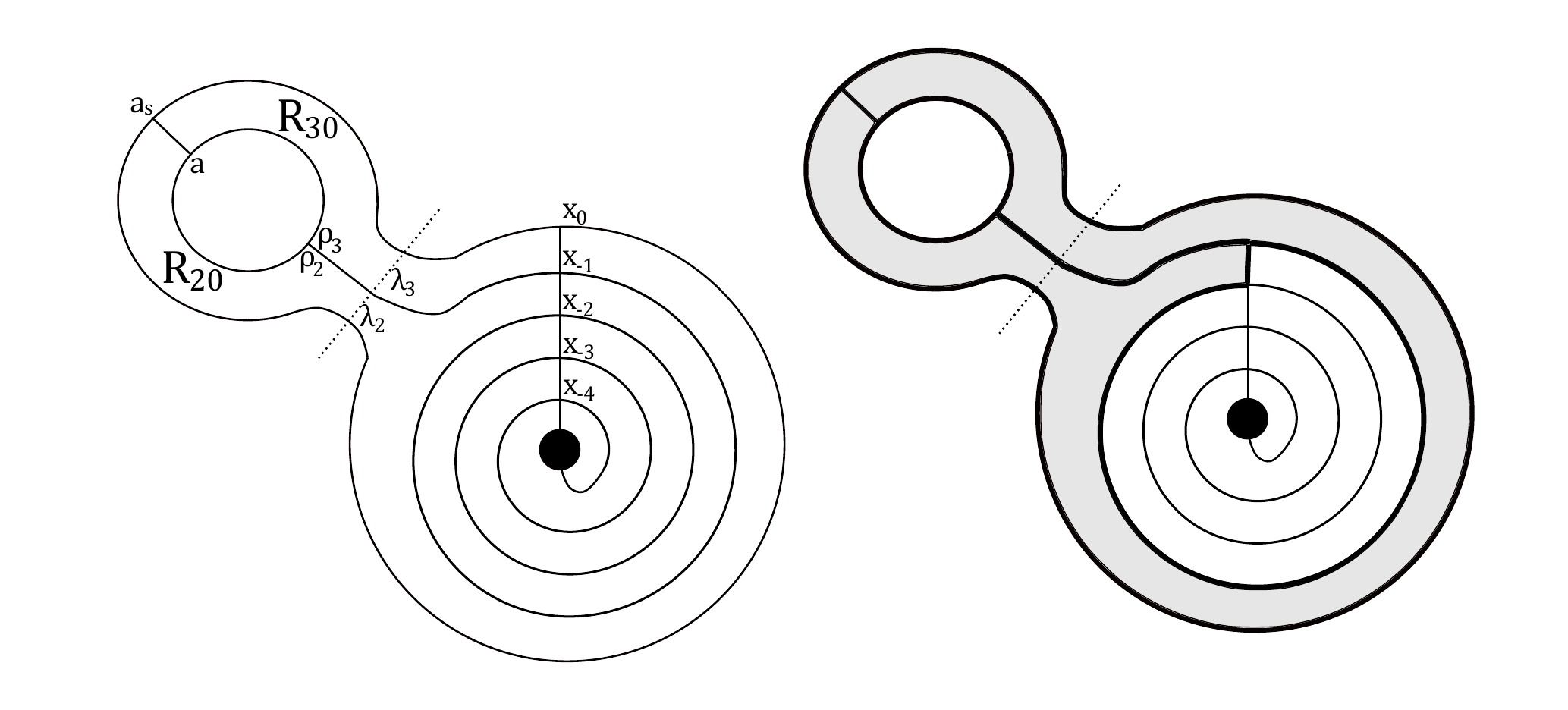}
\caption{The right figure shows the labelling of Lemma~\ref{lem:dilation1}. The dashed line represents $\partial C$. The upper left part of $\partial C$ will play a role of type-$A$ and the bottom right part type-$D$. The left figure implies that the domain $D$ can be regarded as an annulus.}
\label{fig:dilation1}
\end{center}
\end{figure}

\begin{rem}
The domain $D$ can be regarded as an annulus, whose inner boundary consists of $\boldsymbol{\alpha}$-curves and a segment of $\boldsymbol{\beta}$-curve on the winding region of the $\Sigma_K$ part of the diagram. On the other hand, the outer boundary solely consists of $\boldsymbol{\alpha}$-curves. However, when making a cut along $\boldsymbol{\beta}$-curve from $\mathbf{a_s}$, we obtain a holomorphic involution of $D$ interchanging the inner and outer boundaries of the annulus. This also proves Lemma~\ref{lem:dilation1}.
\end{rem}

{\bf Type 2 domain}.  We study the following domains. For $\mathbf{x} \in \mathfrak{S}_K$,  
\begin{itemize}
  \item the domains of $\overline{\phi}_i \in \pi_2 ( \mathbf{x}_i, \mathbf{x}_{i+1} )$, $i \geq 1$, glued to $R_{01} + R_{02}$;
  \item the domain of $\overline{\phi}_0 \in \pi_2 ( \mathbf{x}_0, \mathbf{x}_1 )$ glued to $R_{02}$;
  \item the domain of $\overline{\phi}_0 \in \pi_2 ( \mathbf{x}_0, \mathbf{x}_1 )$ glued to $R_{02} + R_{13}$.
\end{itemize}

The second domain is rectangular, therefore we have $\# \mathcal{M} ( \mathbf{x}_0 \mathbf{bc}, \mathbf{x}_1 \mathbf{a_s d} ; \sigma_2 ) =1$. The third domain could possibly contribute to the term $\rho_1 \sigma_{23} \cdot \mathbf{x}_1 \mathbf{a_s d}$ in $\delta^1 ( \mathbf{x}_0 \mathbf{ad} )$, but the idempotent rule does not allow this term to exist. Thus, we only need to study the first domain.

\begin{lem}
For $\mathbf{x} \in \mathfrak{S}_K$, the moduli space $\mathcal{M} ( \mathbf{x}_i \mathbf{a_s d}, \mathbf{x}_{i+1} \mathbf{a_s d} ; (\sigma_1, \sigma_2) )$, $i \geq 1$, has modulo two count one.
\label{lem:dilation2}
\end{lem}
\begin{proof}
The proof of this lemma also follows the same trick that we have used in Lemma~\ref{lem:dilation1}. For a homology class $\overline{\phi}_i \in \pi_2 (\mathbf{x}_i, \mathbf{x}_{i+1})$, the domain of $\overline{\phi}_i$ glued to $R_{01} + R_{02}$ will be denoted by $D$.  Again, divide $D$ along $\partial C \cap D$ and we have two Reeb chords $\lambda_0$ and $\lambda_1$ on $\partial C \cap D$. See the illustration in Figure~\ref{fig:dilation2}. Then the domain $D$ is decomposed into
\begin{itemize}
  \item the domain of $\overline{\phi}$, whose only valid interpretation is a bigon with moduli space $\mathcal{M} ( \mathbf{x}_i, \mathbf{x}_{i+1} ; (\lambda_0, \lambda_1) )$
  \item the rectangular domain $R_{01}$ with moduli space $\mathcal{M} ( \mathbf{a_s d}, \mathbf{bc} ; \lambda_0 )$
  \item the rectangular domain $R_{02}$ with moduli space $\mathcal{M} ( \mathbf{bc}, \mathbf{a_s d} ; \lambda_1 )$.
\end{itemize} 
Again the same argument in Lemma~\ref{lem:dilation1} proves that the moduli space of the fibered product has a unique point.
\end{proof}

\begin{figure}
\begin{center}
\includegraphics[scale=0.7]{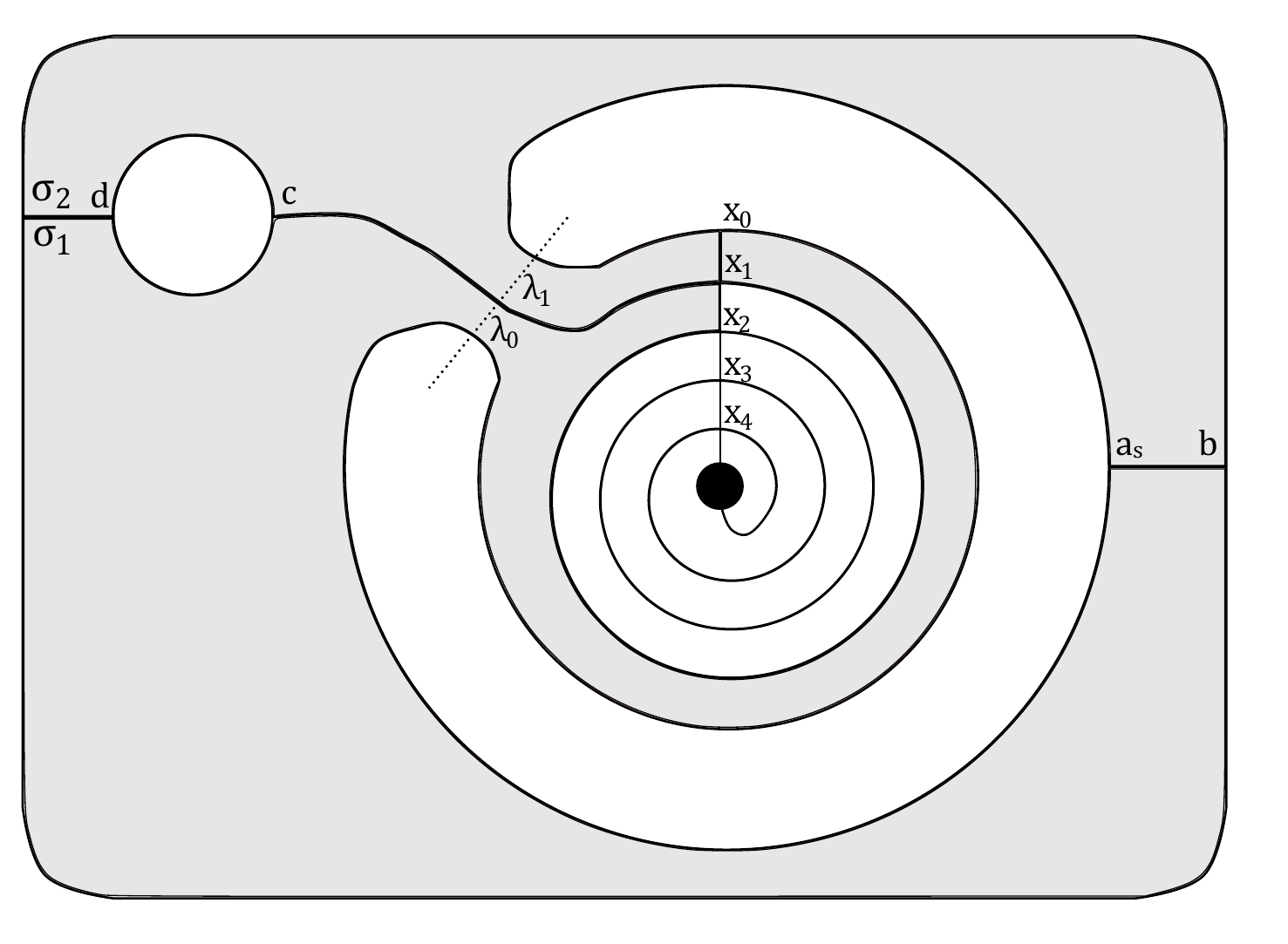}
\caption{The domain of $\overline{\phi}_i \in \pi_2 (\mathbf{x}_i, \mathbf{x}_{i+1} )$ glued to $R_{01} + R_{02}$. Again, the dashed line represents a part of $\partial C$. }
\label{fig:dilation2}
\end{center}
\end{figure}

{\bf Type 3 domain}. The domains for consideration are
\begin{itemize}
  \item the domain of $\psi_i \in \pi_2 (\mathbf{x}_i, \mathbf{x}_{i-1})$, $i \geq 2$, glued to $R_{20} + R_{30}$
  \item the domain of $\psi_1 \in \pi_2 (\mathbf{x}_1, \mathbf{x}_0)$, glued to $R_{20} + R_{30} + R_{01}$
  \item the domain of $\psi_1 \in \pi_2 (\mathbf{x}_1, \mathbf{x}_0)$, glued to $R_{20} + R_{30} + R_{01} + R_{02}$
\end{itemize}

We will study the first domains. First recall that $\# \mathcal{M} ( \mathbf{x}_i, \mathbf{x}_{i-1} ; (\lambda_2, \lambda_3) ) =1$, by~\cite[Lemma 11.48]{LOT08}. Then by the same argument in Lemma~\ref{lem:dilation1}, we can easily observe $\# \mathcal{M} ( \mathbf{x}_i \mathbf{a_s d}, \mathbf{x}_{i-1} \mathbf{a_s d} ; (\rho_2, \rho_3) ) =1$ for $i \geq 2$. \\

The second domain is dealt in the following Lemma.

\begin{lem}
Let $D$ be the domain obtained by gluing $\psi_2 \in \pi_2 (\mathbf{x}_2, \mathbf{x}_1)$ and $R_{20} + R_{30} + R_{01}$. Then the moduli space $\mathcal{M} (\mathbf{x}_1 \mathbf{a_s d}, \mathbf{x}_0 \mathbf{bc} ; (\rho_2, \rho_3, \sigma_1) )$ of the domains has modulo two count one.
\label{lem:dilation3}
\end{lem}
\begin{proof}
Cutting $\Sigma$ along $\partial C$, we have two diagrams, $\Sigma_K$ and its complement. On $\Sigma_K$, we let $\widetilde{D}$ denote a rectangular region that connects $x_0$ and $x_1$ adjacent to the chord $\lambda_1$. Consider $\Sigma_K \backslash \widetilde{D}$. This domain has not been considered in the proof of \cite[Theorem 11.36]{LOT08}, since in their convention the domain contained the distinguished point $z$. However, by the parallel analysis in \cite[Lemma 11.49]{LOT08}, the only valid interpretation of this domain is $\mathcal{M} ( \mathbf{x}_1, \mathbf{x}_0 ; (\lambda_2, \lambda_3, \lambda_0 ) )$ and its modulo two count is one. Then we glue the domain $\Sigma_K \backslash \widetilde{D}$ to $R_{20} + R_{30} + R_{01}$. The domains $R_{20}, R_{30}$ and $R_{01}$ are all rectangular; therefore we have three moduli spaces
\begin{displaymath}
\mathcal{M} ( \mathbf{a_s d}, \mathbf{a d} ; R_{20} ), \quad \mathcal{M} ( \mathbf{a d}, \mathbf{a_s d} ; R_{30} ), \quad \mathrm{and} \ \mathcal{M} ( \mathbf{a_s d}, \mathbf{bc} ; R_{01} ).
\end{displaymath}
Each of these moduli spaces has a unique point, and by the same trick used in Lemma~\ref{lem:dilation1}, the claim is proved.
\end{proof}

The last domain is a periodic domain, and the domain does not have an appropriate interpretation of expected dimension one. \\

{\bf Type 4 domain}. We will consider the following domains.
\begin{itemize}
  \item the domain of $\overline{\psi}_i \in \pi_2 ( \mathbf{x}_{-i}, \mathbf{x}_{-i+1} )$, $i \geq 2$, glued to $R_{01} + R_{02}$
  \item the domain of $\overline{\psi}_1 \in \pi_2 ( \mathbf{x}_{-1}, \mathbf{x}_0 )$, glued to $R_{20} + R_{01} + R_{02}$
  \item the domain of $\overline{\psi}_1 \in \pi_2 ( \mathbf{x}_{-1}, \mathbf{x}_0 )$, glued to $+R_{13} + R_{20} + R_{01} + R_{02}$.
\end{itemize}
The first domain has a moduli space $\mathcal{M} ( \mathbf{x}_{-i} \mathbf{a_s d}, \mathbf{x}_{-i+1} \mathbf{a_s d} ; \sigma_1, \sigma_2 )$, and by the same consideration of Lemma~\ref{lem:dilation2}, its modulo two count is one. The second domain has a moduli space $\mathcal{M} ( \mathbf{x}_{-1} \mathbf{a_s d}, \mathbf{x}_0 \mathbf{ad} ; \rho_2, \sigma_1, \sigma_2)$, and it has modulo two count is one again by the same argument of Lemma~\ref{lem:dilation3}. The last domain does not count due to the idempotent restriction. \\

{\bf Type 5 domain}. For $\mathbf{x}, \mathbf{y} \in \mathfrak{S}_K$, let us suppose there exists a domain $B$ of the doubly pointed Heegaard diagram $\mathcal{H}_K$ from $\mathbf{x}$ to $\mathbf{y}$, such that $n_z (B) = l$ $n_w (B) = 0$. Then for a homology class $\varphi_i \in \pi_2 ( \mathbf{x}_i, \mathbf{y}_{i-l} )$, we consider the following domains.
\begin{itemize}
  \item the domain of $\varphi_i \in \pi_2 ( \mathbf{x}_i, \mathbf{y}_{i-l} )$, $i > l$
  \item the domain of $\varphi_l \in \pi_2 ( \mathbf{x}_l, \mathbf{y}_0 )$, glued to $R_{01}$
  \item the domain of $\varphi_l \in \pi_2 ( \mathbf{x}_l, \mathbf{y}_0 )$, glued to $R_{01} + R_{02}$ and the domain of $\overline{\phi}_0 \in \pi_2 (\mathbf{x}_0, \mathbf{x}_1 )$ 
  \item the domain of $\varphi_l \in \pi_2 ( \mathbf{x}_l, \mathbf{y}_0 )$, glued to $R_{01} + R_{02} + R_{13}$ and the domain of $\overline{\phi}_0 \in \pi_2 (\mathbf{x}_0, \mathbf{x}_1 )$
\end{itemize} 

The first domain is the exactly same domain considered in~\cite[Lemma 11.48]{LOT08}, so it contributes to the coefficient domain $D_{\emptyset} : \mathbf{x}_i \mathbf{a_s d} \rightarrow \mathbf{x}_{i-l} \mathbf{a_s d}$. The second domain is divided into the domain of $\varphi_l$ and $R_{01}$, and by the standard pairing argument of Lemma~\ref{lem:dilation1}, it results a moduli space $\mathcal{M} ( \mathbf{x}_1 \mathbf{a_s d}, \mathbf{y}_0 \mathbf{bc} ; \sigma_1)$. The third domain is a valid domain only when the length $l$ of the arrow equals one, otherwise it would connect invalid generators. We will study this domain in the following Lemma.

\begin{lem}
If $l=1$, then the moduli space $\mathcal{M} ( \mathbf{x}_0 \mathbf{bc}, \mathbf{y}_0 \mathbf{bc} ; \sigma_{12} )$ has modulo two count one. 
\end{lem} 
\begin{proof}
We dualize the module $M$ in order to take advantage of the $\mathcal{A}_{\infty}$-relation. In particular, the orientation of the boundary of $\Sigma$ is reversed. We decompose the domain into two: the domain of $\varphi_l$ glued to $R_{01}$, and the domain of $\overline{\phi}_0$ glued to $R_{02}$. Each domain corresponds to the $\mathcal{A}_{\infty}$-relation
\begin{displaymath}
m ( \mathbf{x}_0 \mathbf{bc} ,\overline{ \sigma }_2 ) = \mathbf{x}_1 \mathbf{a_s d}, \quad m( \mathbf{x}_1 \mathbf{a_s d} , \overline{ \sigma }_3 ) = \mathbf{y}_0 \mathbf{bc}.
\end{displaymath}
(Here the algebra elements with overlines emphasize reversed orientation.) The $\mathcal{A}_{\infty}$ relation of type-$AA$ module gives
\begin{eqnarray*}
0 & = & m^2 ( \mathbf{x}_0 \mathbf{bc} , \overline{\sigma}_2 , \overline{\sigma}_3 ) = m ( \mathbf{x}_1 \mathbf{a_s d}, \overline{\sigma}_3 ) + m ( \mathbf{x}_0 \mathbf{bc}, \overline{\sigma}_{23} ) \\
& = & \mathbf{y}_0 \mathbf{bc} + m ( \mathbf{x}_0 \mathbf{bc}, \overline{\sigma}_{23} ).
\end{eqnarray*}
Reversing $\overline{\sigma}_{23}$ will give $\sigma_{12}$, which proves the claim.
\end{proof}
However, the idempotent rule prohibits the above moduli space from contributing to $\delta^1$. \\

The fourth domain is studied in a similar manner.

\begin{lem}
If $l=1$, then $\mathcal{M} ( \mathbf{x}_0 \mathbf{ad} , \mathbf{y}_0 \mathbf{bc} ; \rho_1, \sigma_{123} )$ has modulo two count one.
\end{lem}
\begin{proof}
Again, we dualize the module. Then there are following correspondences between moduli spaces and $\mathcal{A}_{\infty}$-relations.
\begin{itemize}
  \item $\mathcal{M} ( \mathbf{x}_0 \mathbf{ad}, \mathbf{x}_0 \mathbf{bc} ; \rho_1, \sigma_3 )$, obtained by the obvious rectangular domain in $\Sigma_{\mathcal{L}} \backslash C$, corresponding to $m ( \mathbf{x}_0 \mathbf{ad} , \overline{\rho}_3, \overline{\sigma}_1 ) = \mathbf{x}_0 \mathbf{bc}$
  \item $\mathcal{M} ( \mathbf{x}_0 \mathbf{bc}, \mathbf{x}_1 \mathbf{a_s d} ; \sigma_2 )$, corresponding to $m ( \mathbf{x}_0 \mathbf{bc}, \overline{\sigma}_2 ) = \mathbf{x}_1 \mathbf{a_s d}$
  \item $\mathcal{M} ( \mathbf{x}_1 \mathbf{a_s d}, \mathbf{y}_0 \mathbf{bc} ; \sigma_1 )$, corresponding to $m ( \mathbf{x}_1 \mathbf{a_s d}, \overline{\sigma}_3 ) = \mathbf{y}_0 \mathbf{bc}$
\end{itemize}
Through the combination of the above $\mathcal{A}_{\infty}$-relation and by reversing the boundary orientation, we prove the claim.
\end{proof}

{\bf Type 6 domain}. Similarly, for $\mathbf{x}, \mathbf{y} \in \mathfrak{S}_K$ let us suppose there exists a domain $B$ of the doubly pointed Heegaard diagram $\mathcal{H}_K$ from $\mathbf{x}$ to $\mathbf{y}$ such that $n_z (B) = 0$ $n_w (B) = l$. Then we have homology classes $\overline{\varphi}_i \in \pi_2 (\mathbf{x}_{-i}, \mathbf{y}_{-i+l})$, $i>l$. Then these domains result the following domains $\Sigma$ by appropriate gluing.
\begin{itemize}
  \item the domain of $\overline{\varphi}_i \in \pi_2 ( \mathbf{x}_{-i}, \mathbf{y}_{-i+l} )$, $i > l$
  \item the domain of $\overline{\varphi}_l \in \pi_2 ( \mathbf{x}_l, \mathbf{y}_0 )$, glued to $R_{20}$
  \item the domain of $\overline{\varphi}_l \in \pi_2 ( \mathbf{x}_l, \mathbf{y}_0 )$, glued to $R_{20} + R_{30}$ and the domain of $\phi_0 \in \pi_2 (\mathbf{x}_0 , \mathbf{x}_{-1} )$
  \item the domain of $\overline{\varphi}_l \in \pi_2 ( \mathbf{x}_l, \mathbf{y}_0 )$, glued to $R_{20} + R_{30} + R_{13}$ and the domain of $\phi_0 \in \pi_2 (\mathbf{x}_0 , \mathbf{x}_{-1} )$
\end{itemize} 
The analysis of type 5 domains are similar to the type 5 domains, thus we only list the moduli spaces of modulo two count one instead of repeating the same argument again. 
\begin{itemize}
  \item $\mathcal{M} ( \mathbf{x}_{-1} \mathbf{a_s d} , \mathbf{y}_{-i+l} \mathbf{a_s d} )$, $i >l$
  \item $\mathcal{M} ( \mathbf{x}_l \mathbf{a_s d}, \mathbf{y}_0 \mathbf{ad} ; \rho_2 )$
  \item $\mathcal{M} ( \mathbf{x}_0 \mathbf{ad}, \mathbf{y}_0 \mathbf{bc} ; \rho_{123} \sigma_3 )$, if $l=1$
\end{itemize}

{\bf Other domains}. There are two other domains that are not included in the discussion above. The easier one is the domain $R_{13}$, which is rectangular and has an moduli space $\mathcal{M} ( \mathbf{x}_0 \mathbf{ad}, \mathbf{x}_0 \mathbf{bc} ; \rho_1, \sigma_3 )$ of modulo two count one. \\

Then we turn to the domain that contributes to the algebra element $\rho_{123} \sigma_{123}$. Let $D$ be a domain of $\Sigma$ which has multiplicity one on all regions but the domain contains the arc $\mathbf{z}$. Obviously the domain contributes to the differential from $\mathbf{x}_0 \mathbf{ad}$ to $\mathbf{x}_0 \mathbf{bc}$ for each $x \in \mathfrak{S}_K$. The domain $D$ has the following three interpretations.
\begin{itemize}
  \item $\mathcal{M} ( \mathbf{x}_0 \mathbf{ad}, \mathbf{x}_0 \mathbf{bc} ; \rho_{123}, \sigma_1, \sigma_2, \sigma_3 )$
  \item $\mathcal{M} ( \mathbf{x}_0 \mathbf{ad}, \mathbf{x}_0 \mathbf{bc} ; \rho_1, \rho_2, \rho_3, \sigma_{123} )$
  \item $\mathcal{M} ( \mathbf{x}_0 \mathbf{ad}, \mathbf{x}_0 \mathbf{bc} ; \rho_1, \rho_2, \rho_3, \sigma_1, \sigma_2, \sigma_3 )$
\end{itemize}
We claim that all of these moduli spaces have modulo two count one, and this proves the existence of the term $\rho_{123} \sigma_{123} \mathbf{x}_0 \mathbf{bc}$ in $\delta^1 \mathbf{x}_0 \mathbf{a_s d}$. \\

The moduli spaces of the first two domains can be dealt within the $\mathcal{A}_{\infty}$-relation by the dualizing technique. Explicitly, the combination of moduli spaces
\begin{displaymath}
\mathcal{M} ( \mathbf{x}_0 \mathbf{ad}, \mathbf{x}_{-1} \mathbf{a_s d} ; \overline{\rho}_1), \ \mathcal{M} ( \mathbf{x}_{-1} \mathbf{a_s d}, \mathbf{x}_0 \mathbf{ad} ; \overline{\rho}_2, \overline{\sigma}_3, \overline{\sigma}_2 ), \ \mathcal{M} ( \mathbf{x}_0 \mathbf{ad}, \mathbf{x}_0 \mathbf{bc} ; \overline{\rho}_3, \overline{\sigma}_1)
\end{displaymath}
(note that the first one is considered in type 1 domain, the second in type 4 domains, and the last one in the paragraph above) will result
\begin{displaymath}
\# \mathcal{M} ( \mathbf{x}_0 \mathbf{ad}, \mathbf{x}_0 \mathbf{bc} ; \rho_{123}, \sigma_1, \sigma_2, \sigma_3 ) =1
\end{displaymath}
modulo two. Similarly, the moduli space
\begin{displaymath}
\mathcal{M} ( \mathbf{x}_0 \mathbf{ad}, \mathbf{x}_0 \mathbf{bc} ; \overline{\rho}_3, \overline{\sigma}_1), \ \mathcal{M} ( \mathbf{x}_0 \mathbf{bc}, \mathbf{x}_1 \mathbf{a_s d} ; \overline{\sigma}_2 ), \ \mathcal{M} ( \mathbf{x}_1 \mathbf{a_s d}, \mathbf{x}_0 \mathbf{bc} ; \overline{\rho}_2, \overline{\rho}_1, \overline{\sigma}_3 )
\end{displaymath}
(the second moduli space is considered in type 2 domain, and the last moduli space in type 3 domain) will give us
\begin{displaymath}
\# \mathcal{M} ( \mathbf{x}_0 \mathbf{ad}, \mathbf{x}_0 \mathbf{bc} ; \rho_1, \rho_2, \rho_3, \sigma_{123} ) =1
\end{displaymath}
modulo two.\\

Now, we turn to the moduli space $\mathcal{M} ( \mathbf{x}_0 \mathbf{ad}, \mathbf{x}_0 \mathbf{bc} ; \rho_1, \rho_2, \rho_3, \sigma_1, \sigma_2, \sigma_3 )$. First, let us cut open the domain $D$ along the curve $C$. Then we get two components $\Sigma_K$ and $B$ of $D$. Recall that $\Sigma_K$ is the standard bordered Heegaard diagram of the knot $K$ complement with framing $-n$, thus $B$ can be considered as a complement of $\Sigma_K$ in $D$. Considering the domain $B$ as a domain in the diagram $\Sigma_{\mathcal{L}}$ (see Figure~\ref{fig:connectedsum}), we can again decompose $B$ into three smaller domains, say,
\begin{displaymath}
B_1 : = R_{13}, \quad B_2 : = R_{20} + R_{02}, \quad \textrm{and} \quad B_3 : = R_{30} + R_{01}.
\end{displaymath}
Each rectangular domain $B_i$ can be associated to the following moduli spaces;
\begin{eqnarray*}
\mathcal{M} ( \mathbf{ad}, \mathbf{bc} ; B_1 ) & = & \mathcal{M} ( \mathbf{ad}, \mathbf{bc} ; \rho_1, \sigma_3 ), \\
\mathcal{M} ( \mathbf{bc}, \mathbf{ad} ; B_2 ) & = & \mathcal{M} ( \mathbf{bc}, \mathbf{ad} ; \rho_2, \sigma_2, \lambda_{12} ), \\
\mathcal{M} ( \mathbf{ad}, \mathbf{bc} ; B_3 ) & = & \mathcal{M} ( \mathbf{ad}, \mathbf{bc} ; \rho_3, \sigma_1, \lambda_{30} ).
\end{eqnarray*}
Then observe the domain $B_1 + B_2$ has a moduli space
\begin{displaymath}
\mathcal{M} ( \mathbf{bc}, \mathbf{bc} ; B_1 + B_2) = \mathcal{M} ( \mathbf{bc}, \mathbf{bc} ; \rho_1, \rho_2, \sigma_2, \sigma_3, \lambda_{12} ).
\end{displaymath} 
This moduli space can be interpreted as an annulus, whose outer boundary consists of $\boldsymbol{\alpha}$ and $\boldsymbol{\beta}$ curves and inner boundary consists of $\boldsymbol{\alpha}$ curve only. By making a cut along $\widetilde{\beta}_2$ from $c$, it follows that the moduli space is transversely cut out and has an odd number of points. \\

On the other hand, the domain $\Sigma_K$ cannot have any corner in its interior, and it can be regarded as a boundary degeneration which was introduced in~\cite[Chapter 11]{LOT08}. By the \emph{tautological correspondence}, for a sequence $\overrightarrow{\lambda} = ( \lambda_{12}, \lambda_{30} )$ we may consider a $J$-holomorphic map
\begin{displaymath}
\phi : \mathbb{H} \backslash \{ t_1, t_2 \} \rightarrow \mathrm{Sym}^{g-1} (\Sigma_K).
\end{displaymath}
where $\mathbb{H}$ be the upper half-plane, and $\phi$ is asymptotic to $\mathbf{x}$ at $\infty$ and to a chord $\lambda_{12}$ at $t_1$, $\lambda_{30}$ at $t_2$. Then by~\cite[Proposition 11.34]{LOT08}, the moduli space $\mathcal{M}^{ [ \Sigma_K ] } ( \mathbf{x} ; \overrightarrow{\lambda} )$ that contains $\phi$ is transversely cut out and have an odd number of points. Abusing the notation, for a holomorphic curve (in cylindrical setting) $\phi \in \mathcal{M}^{ [ \Sigma_K ] } ( \mathbf{x} ; \overrightarrow{\lambda} )$, the height difference between two chords $\mathrm{ev}_{\lambda_{12},\lambda_{30}}$ is a positive real number $t_0$. Choose representatives $v_1 \in \mathcal{M} ( \mathbf{ad}, \mathbf{bc} ; \rho_3, \sigma_1, \lambda_{30} )$ and $v_2 \in \mathcal{M} ( \mathbf{bc}, \mathbf{bc} ; \rho_1, \rho_2, \sigma_2, \sigma_3, \lambda_{12} )$, and move these curves so that the difference of the $\mathbb{R}$-coordinates of $v_1$ and $v_2$ is $t_0$. By the standard pairing theorem argument with the time dilation, we can conclude that $\mathcal{M} ( \mathbf{x}_0 \mathbf{ad}, \mathbf{x}_0 \mathbf{bc} ; \rho_1, \rho_2, \rho_3, \sigma_1, \sigma_2, \sigma_3 )$ also has modulo two count one. \\

We close this section by summarizing the discussion so far, in terms of the simplified bases. For aesthetic reasons we write $\mathbf{x}_0 := \mathbf{x}_0 \mathbf{ad}$, $\mathbf{x}_i := \mathbf{x}_i \mathbf{a_s d}$ for $i \neq 0$, and $\mathbf{x}_{\infty} : = \mathbf{x}_0 \mathbf{bc}$. The length of the unstable chain is deduced completely analogous to the~\cite[Theorem 11.26]{LOT08}.

\begin{prop}
Let $CFK^-(K)$ be a model for a reduced chain complex for a knot $K \subset S^3$. Then for sufficiently large interger $n$, the type-$DD$ module $M  = \widehat{CFDD} ( \mathcal{H}_{\mathcal{L}_K} (n) )$ can be derived from $CFK^-(K)$ by the following procedure. \\
For each $\mathbf{x} \in CFK^-(K)$, we have the following elements:
\begin{itemize}
  \item $\mathbf{x}_0 \in \imath_1 \jmath_1 M$ and $\mathbf{x}_{\infty} \in \imath_2 \jmath_2 M$
  \item $\mathbf{x}_i \in \imath_2 \jmath_1 M$, $i = -n/2, \cdots, -1, 1, \cdots, n/2$.
  \item The differential between these elements is
  \begin{displaymath}
  \xymatrix@C=0.75cm{
  & \cdots \ar@/^/[r]^{\sigma_{12}} & \mathbf{x}_{-2} \ar@/^/[r]^{\sigma_{12}} \ar@/^/[l]^{\rho_{23}} & \mathbf{x}_{-1} \ar@/^/[r]^{\rho_2 \sigma_{12}} \ar@/^/[l]^{\rho_{23}} & \mathbf{x}_0 \ar[rr]^{\rho_1 \sigma_3 + \rho_{123} \sigma_{123} } \ar@/^/[l]^{\rho_3} & & \mathbf{x}_{\infty} \ar@/^/[r]^{\sigma_2} & \mathbf{x}_1 \ar@/^/[r]^{\sigma_{12} } \ar@/^/[l]^{\rho_{23} \sigma_1} & \mathbf{x}_2 \ar@/^/[r]^{\sigma_{12}} \ar@/^/[l]^{\rho_{23}} & \cdots \ar@/^/[l]^{\rho_{23}}
  }
  \end{displaymath}
\end{itemize}
Let $\{ \mathbf{x}^k \}$ be a vertically simplified basis with $\mathbf{x}^0$ being the distinguished element. For the vertical arrow of length $l$ from $\mathbf{x}^j$ to $\mathbf{x}^{j+1}$, the differential between the associated elements is
\begin{displaymath}
\xymatrix{
\mathbf{x}^j_0 \ar[rr]^{\rho_1 \sigma_3 + \rho_{123} \sigma_{123}} & &  \mathbf{x}^j_{\infty} \ar@/^/[r]^{\sigma_2} & \mathbf{x}^j_1 \ar@/^/[r]^{\sigma_{12} } \ar@/^/[l]^{\rho_{23} \sigma_1} & \cdots \ar@/^/[r]^{\sigma_{12} } \ar@/^/[l]^{\rho_{23}} & \mathbf{x}^j_l \ar@/^/[r]^{\sigma_{12} } \ar@/^/[l]^{\rho_{23}} \ar[d]^{\sigma_1} & \mathbf{x}^j_{l+1} \ar@/^/[r]^{\sigma_{12} } \ar@/^/[l]^{\rho_{23}} \ar[d]^1 & \cdots \ar@/^/[l]^{\rho_{23}} \\
& & & \mathbf{x}^{j+1}_0 \ar[rr]^{\rho_1 \sigma_3 + \rho_{123} \sigma_{123}} & & \mathbf{x}^{j+1}_{\infty} \ar@/^/[r]^{\sigma_2} & \mathbf{x}^{j+1}_1 \ar@/^/[r]^{\sigma_{12} } \ar@/^/[l]^{\rho_{23} \sigma_1} & \cdots. \ar@/^/[l]^{\rho_{23}}
}
\end{displaymath}
In particular, if $l=1$, then there exists an additional differential $\xymatrix{ \mathbf{x}^j_0 \ar[r]^{\rho_1 \sigma_{123}} & \mathbf{x}^{j+1}_{\infty} }$. \\
On the other hand, let $\{ \mathbf{y}^k \}$ be a horizontally simplified basis with $\mathbf{y}^0$ being the distinguished element. For the horizontal arrow of length $l$ from $\mathbf{y}^j$ to $\mathbf{y}^{j+1}$, the differential between the associated elements is
\begin{displaymath}
\xymatrix{
\mathbf{y}^j_{\infty} & &  \mathbf{y}^j_0 \ar[ll]_{\rho_1 \sigma_3 + \rho_{123} \sigma_{123}} \ar@/^/[r]^{\rho_3} & \mathbf{y}^j_{-1} \ar@/^/[r]^{\rho_{23} } \ar@/^/[l]^{\rho_2 \sigma_{12} } & \cdots \ar@/^/[r]^{\rho_{23} } \ar@/^/[l]^{\sigma_{12}} & \mathbf{y}^j_{-l} \ar@/^/[r]^{\rho_{23} } \ar@/^/[l]^{\sigma_{12}} \ar[d]^{\rho_2} & \mathbf{y}^j_{-l-1} \ar@/^/[r]^{\rho_{23} } \ar@/^/[l]^{\sigma_{12}} \ar[d]^1 & \cdots \ar@/^/[l]^{\sigma_{12}} \\
& & & \mathbf{y}^{j+1}_{\infty} & & \mathbf{y}^{j+1}_0 \ar[ll]^{\rho_1 \sigma_3 + \rho_{123} \sigma_{123}} \ar@/^/[r]^{\rho_3} & \mathbf{y}^{j+1}_{-1} \ar@/^/[r]^{\rho_{23} } \ar@/^/[l]^{\rho_2 \sigma_{12} } & \cdots. \ar@/^/[l]^{\sigma_{12}}
}
\end{displaymath}
In particular, if $l=1$, then there exists an additional differential $\xymatrix{ \mathbf{y}^j_0 \ar[r]^{\rho_{123} \sigma_3 } & \mathbf{y}^{j+1}_{\infty} }$. \\
Lastly, the \emph{unstable chain} between the two distinguished elements is as follows.
\begin{displaymath}
\xymatrix{
\mathbf{x}^0_{\infty} & & \mathbf{x}^0_0 \ar[ll]_{\rho_1 \sigma_3 + \rho_{123} \sigma_{123} } \ar@/^/[r]^{\rho_3} & \gamma_1 \ar@/^/[l]^{\rho_2 \sigma_{12} } \ar@/^/[r]^{\rho_{23}} & \cdots \ar@/^/[l]^{ \sigma_{12} } \ar@/^/[r]^{\rho_{23}} & \gamma_m \ar@/^/[l]^{ \sigma_{12} } \ar@/^/[r]^{\rho_{23} \sigma_1 } & \mathbf{y}^0_{\infty} \ar@/^/[l]^{\sigma_2} & & \mathbf{y}^0_0, \ar[ll]_{\rho_1 \sigma_3 + \rho_{123} \sigma_{123} }
}
\end{displaymath}
where $\gamma_i \in \imath_2 \jmath_1 M$, and $m = n + 2 \tau(K)$.
\label{prop:main}
\end{prop} 

\bigskip

\begin{proof}[Proof of Theorem~\ref{thm:main}]
It remains to find the homotopy equivalent model without the algebra element 1 of chains introduced in Proposition~\ref{prop:main}. The undesirable cancelling pairs can easily be removed by the tricks used in the proof of~\cite[Theorem 11.26]{LOT08}. Thus, the chain 
\begin{displaymath}
\xymatrix{
\mathbf{x}^j_0 \ar[rr]^{\rho_1 \sigma_3 + \rho_{123} \sigma_{123}} & &  \mathbf{x}^j_{\infty} \ar@/^/[r]^{\sigma_2} & \mathbf{x}^j_1 \ar@/^/[r]^{\sigma_{12} } \ar@/^/[l]^{\rho_{23} \sigma_1} & \cdots \ar@/^/[r]^{\sigma_{12} } \ar@/^/[l]^{\rho_{23}} & \mathbf{x}^j_l \ar@/^/[r]^{\sigma_{12} } \ar@/^/[l]^{\rho_{23}} \ar[d]^{\sigma_1} & \mathbf{x}^j_{l+1} \ar@/^/[r]^{\sigma_{12} } \ar@/^/[l]^{\rho_{23}} \ar[d]^1 & \cdots \ar@/^/[l]^{\rho_{23}} \\
& & & \mathbf{x}^{j+1}_0 \ar[rr]^{\rho_1 \sigma_3 + \rho_{123} \sigma_{123}} & & \mathbf{x}^{j+1}_{\infty} \ar@/^/[r]^{\sigma_2} & \mathbf{x}^{j+1}_1 \ar@/^/[r]^{\sigma_{12} } \ar@/^/[l]^{\rho_{23} \sigma_1} & \cdots. \ar@/^/[l]^{\rho_{23}}
}
\end{displaymath}
is homotopy equivalent to
\begin{displaymath}
\xymatrix{
\mathbf{x}^j_0 \ar[rr]^{\rho_1 \sigma_3 + \rho_{123} \sigma_{123}} & &  \mathbf{x}^j_{\infty} \ar@/^/[r]^{\sigma_2} & \mathbf{x}^j_1 \ar@/^/[r]^{\sigma_{12} } \ar@/^/[l]^{\rho_{23} \sigma_1} & \cdots \ar@/^/[r]^{\sigma_{12} } \ar@/^/[l]^{\rho_{23}} & \mathbf{x}^j_l \ar@/^/[r]^{\sigma_1 } \ar@/^/[l]^{\rho_{23}} & \mathbf{x}^{j+1}_{\infty}  \ar@/^/[l]^{\rho_{23} \sigma_2} & & \mathbf{x}^{j+1}_0, \ar[ll]^{\rho_1 \sigma_3 + \rho_{123} \sigma_{123}}
}
\end{displaymath}
and the chain
\begin{displaymath}
\xymatrix{
\mathbf{y}^j_{\infty} & &  \mathbf{y}^j_0 \ar[ll]_{\rho_1 \sigma_3 + \rho_{123} \sigma_{123}} \ar@/^/[r]^{\rho_3} & \mathbf{y}^j_{-1} \ar@/^/[r]^{\rho_{23} } \ar@/^/[l]^{\rho_2 \sigma_{12} } & \cdots \ar@/^/[r]^{\rho_{23} } \ar@/^/[l]^{\sigma_{12}} & \mathbf{y}^j_{-l} \ar@/^/[r]^{\rho_{23} } \ar@/^/[l]^{\sigma_{12}} \ar[d]^{\rho_2} & \mathbf{y}^j_{-l-1} \ar@/^/[r]^{\rho_{23} } \ar@/^/[l]^{\sigma_{12}} \ar[d]^1 & \cdots \ar@/^/[l]^{\sigma_{12}} \\
& & & \mathbf{y}^{j+1}_{\infty} & & \mathbf{y}^{j+1}_0 \ar[ll]^{\rho_1 \sigma_3 + \rho_{123} \sigma_{123}} \ar@/^/[r]^{\rho_3} & \mathbf{y}^{j+1}_{-1} \ar@/^/[r]^{\rho_{23} } \ar@/^/[l]^{\rho_2 \sigma_{12} } & \cdots. \ar@/^/[l]^{\sigma_{12}}
}
\end{displaymath}
is homotopy equivalent to
\begin{displaymath}
\xymatrix{
\mathbf{y}^j_{\infty} & &  \mathbf{y}^j_0 \ar[ll]_{\rho_1 \sigma_3 + \rho_{123} \sigma_{123}} \ar@/^/[r]^{\rho_3} & \mathbf{y}^j_{-1} \ar@/^/[r]^{\rho_{23} } \ar@/^/[l]^{\rho_2 \sigma_{12} } & \cdots \ar@/^/[r]^{\rho_{23} } \ar@/^/[l]^{\sigma_{12}} & \mathbf{y}^j_{-l} \ar@/^/[r]^{\rho_2 } \ar@/^/[l]^{\sigma_{12}} & \mathbf{y}^{j+1}_0 \ar@/^/[l]^{\rho_3 \sigma_{12}} \ar[rr]^{\rho_1 \sigma_3 + \rho_{123} \sigma_{123}} & & \mathbf{y}^{j+1}_{\infty}.
}
\end{displaymath}
These reductions both hold even if the length of the horizontal or vertical arrow $l$ equals one, again by the similar trick. For example, if a horizontal arrow has length one, replace $\mathbf{y}^j_{-1}$ by $\widetilde{\mathbf{y}}^j_{-1} := \mathbf{y}^j_{-1} + \rho_{23} \cdot \mathbf{y}^j_{\infty}$. The statement on the unstable chain carries over from Proposition~\ref{prop:main}, thus proving the claim. 
\end{proof}

\begin{figure}
\begin{center}
\begin{displaymath}
\xymatrix{
\mathbf{x} & \mathbf{y} \ar[l] \ar[d] & & & \vdots \ar@/_1pc/[d]_{\sigma_{12}} & \vdots \ar@/^1pc/[d]^{\sigma_{12}} \ar[dl]|1 & & & & \\
& \mathbf{z} & & & \mathbf{x}_{-2} \ar[u]|{\rho_{23}} \ar@/_1pc/[d]_{\sigma_{12}} & \mathbf{y}_{-2} \ar[u]|{\rho_{23}} \ar@/^1pc/[d]^{\sigma_{12}} \ar[dl]|1  & & & & \\
& & & & \mathbf{x}_{-1} \ar[u]|{\rho_{23}} \ar@/_1pc/[d]_{\rho_2 \sigma_{12}} & \mathbf{y}_{-1} \ar[u]|{\rho_{23}} \ar@/^1pc/[d]^{\rho_2 \sigma_{12}} \ar[dl]|{\rho_2}  & & & & \\
& & & & \mathbf{x}_0 \ar[u]|{\rho_3} \ar[dl]_{\rho_1 \sigma_3 + \rho_{123} \sigma_{123} } & \mathbf{y}_0 \ar[u]|{\rho_3} \ar[dr]^{\rho_1 \sigma_3 + \rho_{123} \sigma_{123} } \ar@/^1pc/[dll]^{\rho_{123} \sigma_3} \ar@/_1pc/[ddddr]_{\rho_1 \sigma_{123}} & & & & \\
& & & \mathbf{x}_{\infty} \ar[dl]^{\sigma_2} & & & \mathbf{y}_{\infty} \ar[dr]_{\sigma_2} & & & \\
\cdots \ar@/^1pc/[r]^{\rho_{23}} & \mathbf{x}_2 \ar[l]^{\sigma_{12}} \ar@/^1pc/[r]^{\rho_{23}} & \mathbf{x}_1 \ar[l]^{\sigma_{12}} \ar@/^1pc/[ur]^{\rho_{23} \sigma_1}& & & & & \mathbf{y}_1 \ar[r]_{\sigma_{12}} \ar@/_1pc/[ul]_{\rho_{23} \sigma_1} \ar@/_1pc/[ddl]|{\sigma_1} & \mathbf{y}_2 \ar[r]_{\sigma_{12}} \ar@/_1pc/[l]_{\rho_{23}} \ar[dl]|1 & \cdots \ar@/_1pc/[l]_{\rho_{23}} \ar[dl]|1 \\
 & & & & & & & \mathbf{z}_1 \ar[r]^{\sigma_{12}} \ar@/^1pc/[dl]^{\rho_{23} \sigma_1} & \mathbf{z}_2 \ar[r]^{\sigma_{12}} \ar@/^1pc/[l]^{\rho_{23}} & \cdots \ar@/^1pc/[l]^{\rho_{23}} \\
 & & & & & & \mathbf{z}_{\infty} \ar[ur]^{\sigma_2} & & & \\
 & & & & & \mathbf{z}_0 \ar[ur]_{\rho_1 \sigma_3 + \rho_{123} \sigma_{123} } \ar[d]^{\rho_3} & & & & \\
 & & & & & \mathbf{z}_{-1} \ar[d]^{\rho_{23}} \ar@/^1pc/[u]^{\rho_2 \sigma_{12}} & & & & \\
 & & & & & \mathbf{z}_{-2} \ar[d]^{\rho_{23}} \ar@/^1pc/[u]^{\sigma_{12}} & & & & \\
 & & & & & \vdots \ar@/^1pc/[u]^{\sigma_{12}} & & & &
}
\end{displaymath}
\caption{The diagram describes the type-$DD$ module structure of $\widehat{CFDD} ( \mathcal{H}_{\mathcal{L}_K} (n) )$ obtained directly from Proposition~\ref{prop:main}, where $K$ is the right-handed trefoil knot. The standard model of $CFK^-(K)$ that is used to derive this module is drawn on the top left corner. The elements of the unstable chain are $\mathbf{z}_i$, $i<0$, and $\mathbf{x}_j$, $j>0$. }
\label{fig:trefoil_inv}
\end{center}
\end{figure}
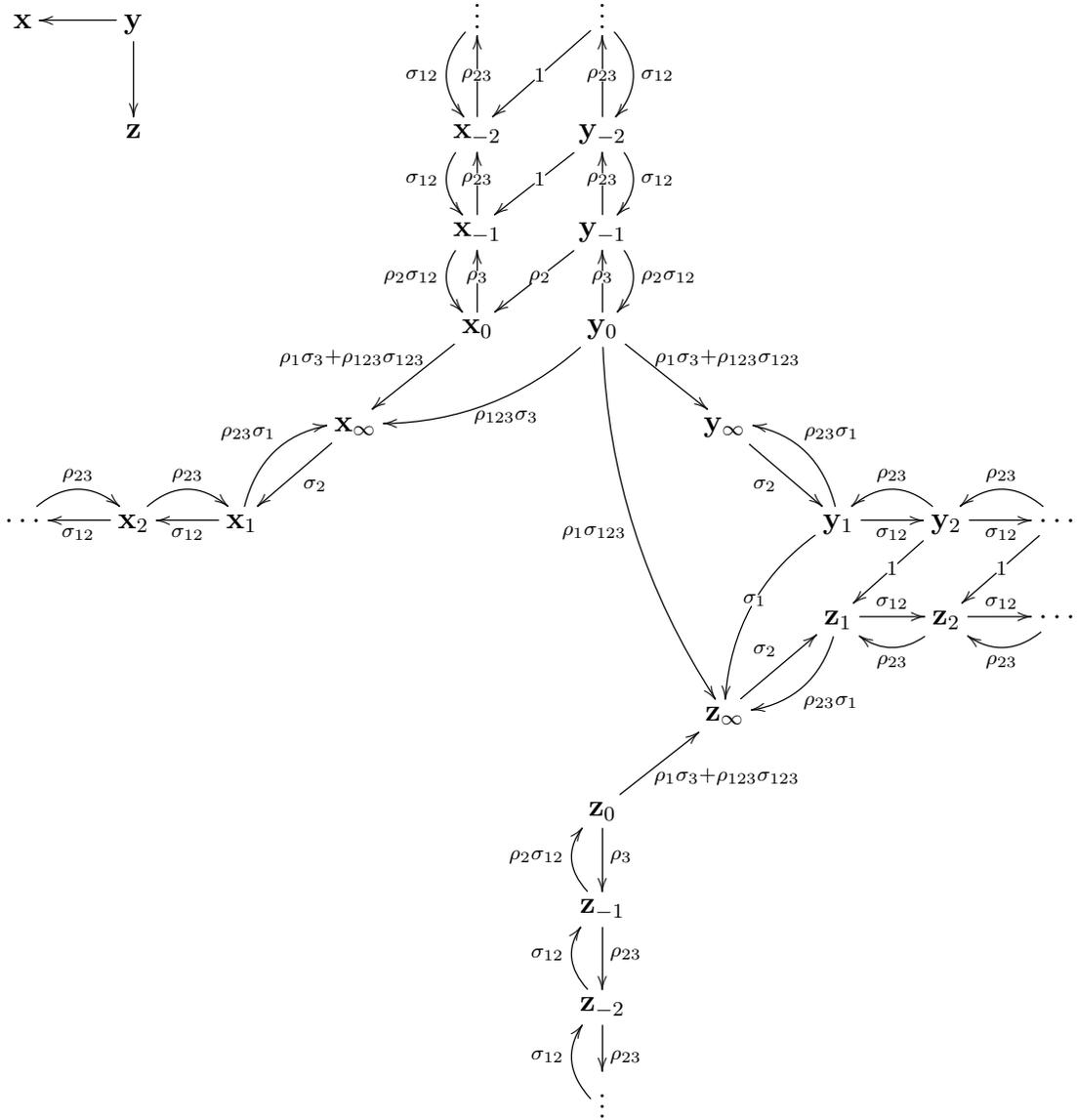

\begin{figure}
\begin{displaymath}
\xymatrix{
 & x_0 \ar[dl]_{\rho_1 \sigma_3 + \rho_{123} \sigma_{123}} \ar@/^1pc/[r]^{\rho_3 \sigma_{12}} & y_{-1} \ar@/^1pc/[r]^{\rho_2 \sigma_{12}} \ar[l]^{\rho_2} & y_0 \ar[l]^{\rho_3} \ar[dr]^{\rho_1 \sigma_3 + \rho_{123} \sigma_{123}} & \\
x_{\infty} \ar@/_1pc/[d]_{\sigma_2} & & & & y_{\infty} \ar[d]_{\sigma_2} \\
z_{-m} \ar@/_1pc/[d]_{\sigma_{12}} \ar[u]_{\rho_{23} \sigma_1} & & & & y_1\ar[d]_{\sigma_1} \ar@/_1pc/[u]_{\rho_{23} \sigma_1} \\
\vdots \ar[u]_{\rho_{23}} & & & & z_{\infty} \ar@/_1pc/[u]_{\rho_{23} \sigma_2} \\
 & \cdots \ar@/_1pc/[r]_{\sigma_{12}} \ar@/^1pc/@{.}[ul] & z_{-1} \ar@/_1pc/[r]_{\rho_2 \sigma_{12}} \ar[l]_{\rho_{23}} & z_0 \ar[l]_{\rho_3} \ar[ur]_{\rho_1 \sigma_3 + \rho_{123} \sigma_{123}} & & 
}
\end{displaymath}
\caption{A diagram obtained after removing all of the algebra element 1 from Figure~\ref{fig:trefoil_inv}.}
\label{fig:trefoil_simp}
\end{figure}
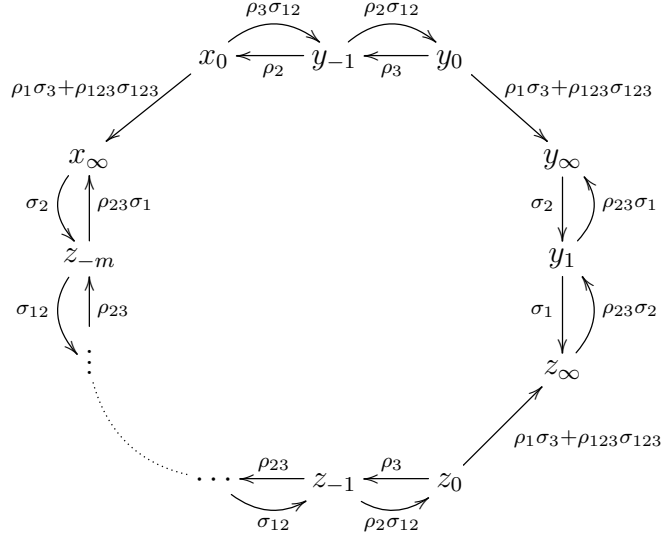

\section{Grading of the meridian complement}
\label{sec:grading}
For a sufficiently large positive integer $n$, the type-$D$ structure of the meridian complement in the integral Dehn surgery manifold $S^3_{-n} (K)$ is easily obtained by taking the $0$-surgery on the left boundary component. The associated type-$A$ module $\widehat{CFA} (\mathcal{H}_0)$ of the $0$-surgery has a single generator $\mathbf{t}$ satisfying the following structure.
\begin{displaymath}
m_{k+1} ( \mathbf{t}, \rho_2, \rho_{12}, \cdots, \rho_{12}, \rho_1 ) = \mathbf{t}.
\end{displaymath} 
Obviously $\mathbf{t}$ cannot be tensored with generators in $\imath_1 \jmath_1 \widehat{CFDD} ( \mathcal{H}_{\mathcal{L}_K} (n) )$. The derived tensor product of the $0$-surgery type-$A$ module with $\widehat{CFDD} ( \mathcal{H}_{\mathcal{L}_K} (n) )$ is a straight-forward computation, so we merely state the result in the following proposition. We omitted tensor with $\mathbf{t}$ from the statement for cosmetic reasons.

\begin{prop}
Under the same assumption in Theorem~\ref{thm:main}, the type-$D$ module $\widehat{CFD} ( S^3_{-n}(K) \backslash \mu_K )$ can be derived from $CFK^-(K)$ by the following procedure. \\
Let $\{ \mathbf{x}^k \}$ be a vertically simplified basis. For a vertical arrow of length $l$ from $\mathbf{x}^j$ to $\mathbf{x}^{j+1}$, the differential between the associated elements is
\begin{displaymath}
\xymatrix{
\mathbf{x}^j_{\infty} \ar[r]^{\sigma_2} & \mathbf{x}^j_1 \ar[r]^{\sigma_{12} } & \cdots \ar[r]^{\sigma_{12} } & \mathbf{x}^j_l \ar[r]^{\sigma_1 } & \mathbf{x}^{j+1}_{\infty}.
}
\end{displaymath}
Let $\{ \mathbf{y}^k \}$ be a horizontally simplified basis. For a horizontal arrow of length $l$ from $\mathbf{y}^j$ to $\mathbf{y}^{j+1}$, the differential between the associated elements is
\begin{displaymath}
\xymatrix{
\mathbf{y}^j_{\infty} & \mathbf{y}^j_{-1} \ar[l]_{\sigma_{123}} & \cdots \ar[l]_{\sigma_{12}} & \mathbf{y}^j_{-l} \ar[r]^{\sigma_3} \ar[l]_{\sigma_{12}} & \mathbf{y}^{j+1}_{\infty}.
}
\end{displaymath}
The unstable chain between the two distinguished elements is as follows.
\begin{displaymath}
\xymatrix{
\mathbf{x}^0_{\infty} & \gamma_1 \ar[l]_{\sigma_{123} } & \cdots \ar[l]_{ \sigma_{12} } & \gamma_m \ar[l]_{ \sigma_{12} } & \mathbf{y}^0_{\infty} \ar[l]_{\sigma_2},
}
\end{displaymath}
where $m=n+2 \tau(K)$.
\label{prop:zerosurgery}
\end{prop}
For simplicity, we will assume that $n$ is an even integer without loss of generality. \\

We often regard the complex $CFK^-(K)$ as a directed graph on a plane with vertices on integral lattice, such that every arrow is pointing either left or downward (disregarding diagonal arrows as in the standard bordered Floer theory). Viewing the type-$D$ module as a directed graph would help understand the structure of a meridional knot in the Dehn surgery manifold, just like the classical knot Floer complex. In case $CFK^-(K)$ has a horizontally and vertically simplified basis, we will view $\widehat{CFD} ( S^3_{-n}(K) \backslash \mu_K )$ as a directed graph on $(q,r)$-plane such that each arrow is labelled as $\sigma_I$, $I \in \{ 1,2,3,12,23,123 \}$. We let the coordinates of generators on the plane obey the following rule.

\begin{itemize}
  \item The coordinates of generators in $\jmath_2 \widehat{CFD} ( S^3_{-n}(K) \backslash \mu_K )$ are determined as follows. Since there is a one-to-one correspondence between $CFK^-(K)$ and $\jmath_2 \widehat{CFD} ( S^3_{-n}(K) \backslash \mu_K )$ (the $\infty$-labelled generators), put $\mathbf{x}_{\infty}$ so that the $(q,r)$-coordinates of $\mathbf{x}_{\infty}$ and the corresponding $\mathbf{x} \in CFK^-(K)$ agree.
  \item Suppose $\mathbf{x}^j, \mathbf{x}^{j+1} \in CFK^-(K)$ have coordinates $(q,r)$ and $(q,r-l)$ and have a vertical arrow of length $l$ between them, then 
  \begin{itemize}
    \item put $\mathbf{x}^j_1$ on $(q,r-\frac{1}{2})$;
    \item put $\mathbf{x}^j_m$ so that $\mathbf{x}^j_{m-1}$ and $\mathbf{x}^j_m$ have the same $q$-coordinates but $r$-coordinate of $\mathbf{x}^j_m$ is one less than that of $\mathbf{x}^j_{m-1}$.
  \end{itemize}
  \item Suppose $\mathbf{y}^j, \mathbf{y}^{j+1} \in CFK^-(K)$ have coordinates $(q,r)$ and $(q-l,r)$ and have a horizontal arrow of length $l$ between them, then 
  \begin{itemize}
    \item put $\mathbf{y}^j_{-1}$ on $(q-\frac{1}{2},r)$;
    \item put $\mathbf{y}^j_{-m}$ so that $\mathbf{y}^j_{-(m-1)}$ and $\mathbf{y}^j_{-m}$ have the same $r$-coordinates but $q$-coordinate of $\mathbf{y}^j_{-m}$ is one less than that of $\mathbf{y}^j_{-(m-1)}$.
  \end{itemize}
  \item Recall that there are two distinguished elements $\mathbf{x}^0$ and $\mathbf{y}^0$, which generate homologies $H_* (C^h)$ and $H_* (C^v)$ respectively. If the distinguished element $\mathbf{x}^0 \in CFK^-(K)$ is on $(q_0, r_0)$, then let the $(q,r)$-coordinates of $\gamma_{\mu}$, $\mu \leq \frac{1}{2} m$ be $(q_0 - \frac{1}{2} - (\mu -1) ,r_0)$. Similarly, if the distinguished element $\mathbf{y}^0 \in CFK^-(K)$ is on $(r_0, q_0)$, then let the $(q,r)$-coordinates of $\gamma_{\mu}$, $\mu > \frac{1}{2} m$ be  $(r_0, q_0 - \frac{1}{2} - (m-\mu) )$. 
\end{itemize}
See Figure~\ref{fig:grading1} for some examples. \\

\begin{figure}
\begin{center}
\includegraphics[scale=0.6]{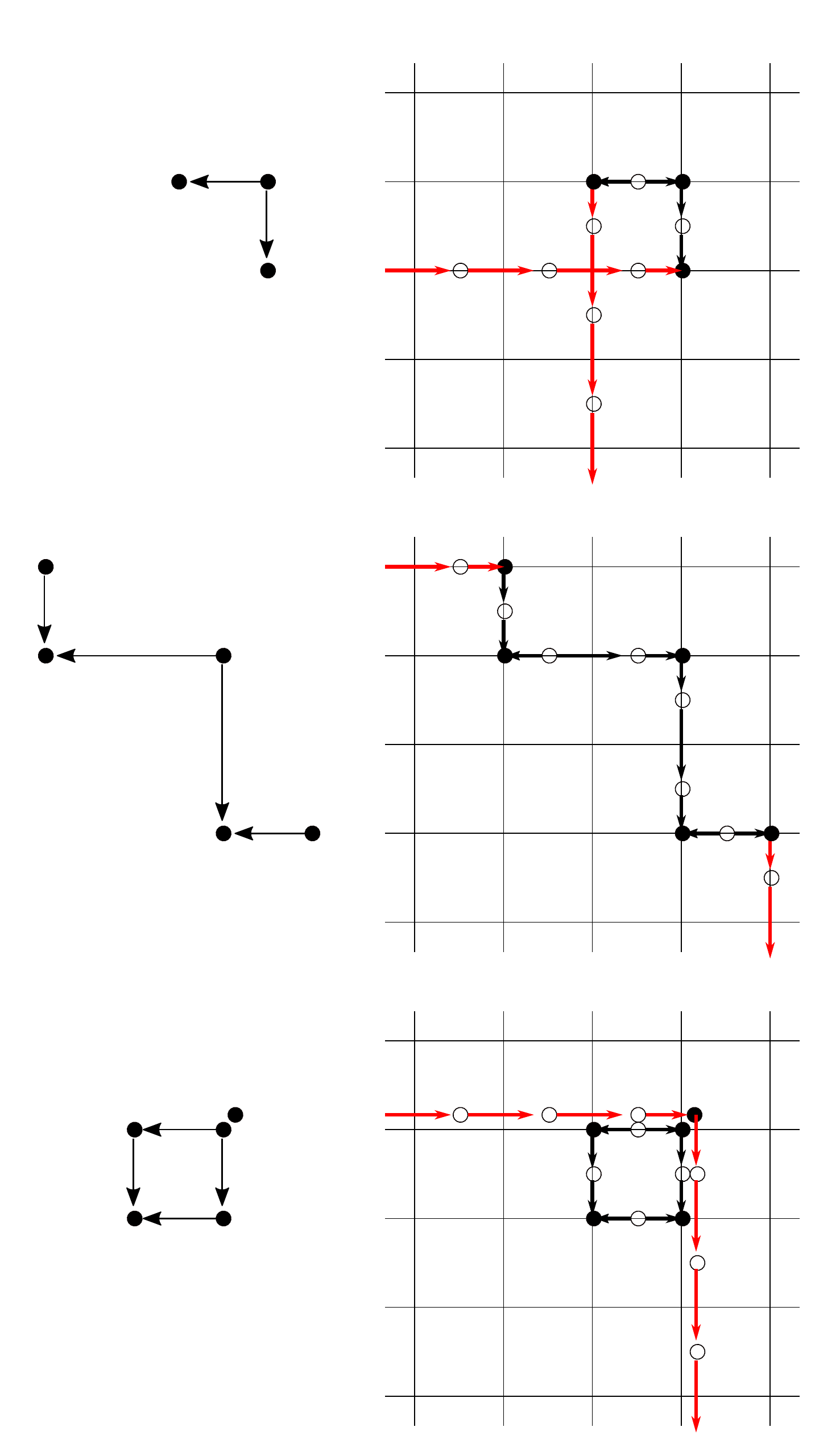}
\caption{From top to bottom, the figures on the left column are the knot Floer complexes of the right-handed trefoil, left-handed (3,4)-torus knot, and the figure-eight knot. The figures on the right column illustrate the type-$D$ module of the meridian complement in the Dehn surgery manifold; the black dots represent generators in $\jmath_2$-idempotent and the white dots represent $\jmath_1$-idempotent. Each arrow is implicitly labelled with $\sigma_I$ according to Proposition~\ref{prop:zerosurgery}. The unstable chain is colored in red.}
\label{fig:grading1}
\end{center}
\end{figure}

Recall that the type-$A$ module $\widehat{CFA} (\mathcal{H}_{\infty})$ associated to the $\infty$-surgery is generated by a single generator $\mathbf{u}$ with the relation
\begin{displaymath}
m_{k+1} (\mathbf{u}, \sigma_3, \sigma_{23}, \cdots, \sigma_{23}, \sigma_2) = \mathbf{u}.
\end{displaymath}
The derived tensor product of $\widehat{CFA} (\mathcal{H}_{\infty})$ with $\widehat{CFD} ( S^3_{-n}(K) \backslash \mu_K )$ will be graded by $P_1 \backslash G / P_2$, where $P_1$ and $P_2$ are subgroups of $G$ generated by gradings of the periodic domains of Heegaard diagrams of $\widehat{CFD} ( S^3_{-n}(K) \backslash \mu_K )$ and  $\widehat{CFA} (\mathcal{H}_{\infty})$ respectively. More precisely, $P_1$ is spanned by $(v;-n,-1)$ for some $v$ and $P_2$ is spanned by $(\frac{3}{2};0,1)$. We will be focusing on the $spin^c$-component of the grading, and the quotient by the action of the Maslov component of this double-coset space will be denoted by $\widetilde{G} : = (P_1 \backslash G / P_2) / \mathbb{Z}$. We have the following decomposition
\begin{eqnarray*}
& & \widehat{CFA} (\mathcal{H}_{\infty}) \boxtimes \widehat{CFD} (S^3_{-n} (K) \backslash \mu_K) \\
& & \cong \widehat{HFK}_* (S^3_{n}(K), \mu_K) =  \bigoplus_{\mathfrak{s}_k \in \mathrm{Spin}^c(S^3_{-n}(K))} \widehat{HFK}_* (S^3_{n}(K), \mu_K, \mathfrak{s}_k).
\end{eqnarray*}
Then, generators in $\widehat{CFA} (\mathcal{H}_{\infty}) \boxtimes \widehat{CFD} (S^3_{-n} (K) \backslash \mu_K)$ that share the same grading in $\widetilde{G}$ belong to the same $spin^c$-summand. \\

Note that the generators in $\widehat{CFA} (\mathcal{H}_{\infty}) \boxtimes \widehat{CFD} (S^3_{-n} (K) \backslash \mu_K)$ have a bijective correspondence to the generators in $\jmath_1 \widehat{CFD} (S^3_{-n} (K) \backslash \mu_K)$. In the following Lemma, we blur the difference between these spaces.

\begin{lem}
Consider $\widehat{CFD} (S^3_{-n} (K) \backslash \mu_K)$ as a directed, $\sigma_I$-labelled graph on $(q,r)$-plane. Then for any integer $k$, the generators in $\jmath_1 \widehat{CFD} (S^3_{-n} (K) \backslash \mu_K)$ that lie on the line $-q+r = k + \frac{1}{2}$ have the same grading in $\widetilde{G}$.
\end{lem}
\begin{proof}
Suppose that generators $a,b \in \widehat{CFD} (S^3_{-n} (K) \backslash \mu_K)$ are connected by a directed edge $\xymatrix{ a \ar[r]^{\sigma_I} & b}$. For simplicity, let $\widetilde{\mathrm{gr}} (a) = (0,0)$. Then the $\widetilde{\mathrm{gr}}(b)$ is,
\begin{displaymath}
\widetilde{\mathrm{gr}}(b) = \left\{
  \begin{array}{ll}
  (-1/2,1/2) & \textrm{if } I = 1 \\
  (-1/2,-1/2) & \textrm{if } I = 2 \\
  (1/2,-1/2) & \textrm{if } I = 3 \\
  (-1,0) & \textrm{if } I = 12 \\
  (0,-1) & \textrm{if } I = 23 \\
  (-1/2,-1/2) & \textrm{if } I = 123
  \end{array}
\right.
\end{displaymath}
To prove the claim, it suffices to keep track of the grading changes in the horizontal and vertical sequences. If $\mathbf{y}^j, \mathbf{y}^{j+1} \in CFK^-(K)$ have a horizontal arrow of length $l$ $\xymatrix{\mathbf{y}^{j+1} & \mathbf{y}^j \ar[l]}$, then the corresponding sequence in $\widehat{CFD} (S^3_{-n} (K) \backslash \mu_K)$ is  
\begin{displaymath}
\xymatrix{
\mathbf{y}^{j+1}_{\infty} & \mathbf{y}^{j}_{-l} \ar[l]_{\sigma_3} \ar[r]^{\sigma_{12}} & \mathbf{y}^j_{-l+1} \ar[r]^{\sigma_{12}} & \cdots \ar[r]^{\sigma_{12}} & \mathbf{y}^j_{-1} \ar[r]^{\sigma_{123}} & \mathbf{y}^j_{\infty} 
}
\end{displaymath}
and it is clear that the first factors of the $\widetilde{G}$-grading of $\jmath_1$-idempotent generators are `increasing' by one from right to left. By similar observation, the first factors of the $\widetilde{G}$-grading of $\jmath_1$-idempotent generators in a vertical sequence are 'decreasing' by one from top to bottom. \\

The proof can be completed by considering the following four cases.
\begin{itemize}
  \item A vertical sequence follows a horizontal sequence
  \item A horizontal sequence follows a vertical sequence
  \item A horizontal and a vertical sequence start at a generator
  \item A horizontal and a vertical sequence end at a generator
\end{itemize}
The figure below illustrates the respective corresponding sequence in the type-$D$ module $\widehat{CFD} (S^3_{-n} (K) \backslash \mu_K )$.
\begin{displaymath}
\xymatrix{
\bullet \ar[d]^{\sigma_2} & a \ar[l]_{\sigma_3} \ar[r]^{\sigma_{12}} & \cdots & & & & \vdots \ar[d]^{\sigma_{12}} \\
b \ar[d]^{\sigma_{12}} & & & & & & d \ar[d]^{\sigma_1} \\ 
\vdots & & & & \cdots \ar[r]^{\sigma_{12}} & c \ar[r]^{\sigma_{123}} & \bullet \\
\cdots \ar[r]^{\sigma_{12}} & e \ar[r]^{\sigma_{123}} & \bullet \ar[d]^{\sigma_2} & & \vdots \ar[d]^{\sigma_{12}} & & \\
& & f \ar[d]^{\sigma_{12}} & & g \ar[d]^{\sigma_1} & & \\
& & \vdots & & \bullet & h \ar[l]_{\sigma_3} \ar[r]^{\sigma{12}} & \cdots \\
}
\end{displaymath} 
The $G$-grading difference between $a$ and $b$ is $\pm(0,1) \in P_2$, thus they are in the same $\widetilde{G}$-grading. Likewise $c$ and $d$ have the same $\widetilde{G}$-grading, too. The $\widetilde{G}$-grading difference between $e$ and $f$ (and $g$ and $h$) are precisely $(1,0)$. This completes the proof.
\end{proof}

\begin{thm}
Let $k \in \mathbb{Z}$. Viewing $\widehat{CFD} (S^3_{-n} (K) \backslash \mu_K)$ as a graph on $(q,r)$-plane, the vector space generated by vertices on the line $L_k := \{ (q,r) | -q+r = k + \frac{1}{2} \}$ is isomorphic to
\begin{displaymath}
H_* \left( \frac{\widehat{CF}(S^3)}{\mathcal{F}(K,k)} \right) \oplus H_* \left( \frac{\widehat{CF}(S^3)}{\mathcal{F}(K,-k-1)} \right).
\end{displaymath}
\end{thm}
\begin{proof}
The line $L_k$ can intersect with horizontal or vertical sequences; if $L_k$ intersects with the vertical sequence, then this implies there is a $\jmath_2$-idempotent generator $\mathbf{y}^j_{\infty}$ endowed with the sequence 
\begin{displaymath}
\xymatrix{ 
\mathbf{y}^j_{\infty} \ar[r]^{\sigma_2} & \mathbf{y}^j_1 \ar[r]^{\sigma_{12}} & \mathbf{y}^j_2 \ar[r]^{\sigma_{12}} & \cdots.
}
\end{displaymath}
Recall that we have a bijective correspondence between generators in $CFK^-(K)$ and $\jmath_2 \widehat{CFD} (S^3_{-n} (K) \backslash \mu_K)$. If the corresponding element $\mathbf{y}^j \in CFK^-(K)$ of $\mathbf{y}^j_{\infty}$ is not the distinguished element in the vertical complex, then there should be $\mathbf{y}^{j+1} \in CFK^-(K)$ with $\partial^v (\mathbf{y}^j) = \mathbf{y}^{j+1}$. Observe that the Alexander filtration of $\mathbf{y}^j$ is greater than $k$ and $\mathbf{y}^{j+1}$ is less than or equal to $k$. Thus, the existence of the intersection of $L_k$ and the vertical sequence imply $\mathbf{y}^j \in H_* ( \widehat{CF}(S^3) / \mathcal{F}(K,k)$. If $\mathbf{y}^j$ is the distinguished element in the vertical complex, i.e., $j=0$, then the Alexander filtration level of $\mathbf{y}^0$ is again greater than $k$ and $\mathbf{y}^0 \in H_* ( \widehat{CF}(S^3) / \mathcal{F}(K,k)$. \\
Similarly, the existence of the intersection of $L_k$ and the horizontal sequence starting at $\mathbf{x}^j_{\infty}$ imply $\mathbf{x}^j \in CFK^-(K)$ has the Alexander filtration greater than $-k-1$ and $\partial^h (\mathbf{x}^j)$ has a filtration level less than or equal to $-k-1$.  
\end{proof}

\begin{proof}[Proof of Theorem~\ref{thm:sub}]
It remains to find $k$ values that admit nonempty intersections between the graph and line $L_k$. Let $\mathbf{x}^0$ and $\mathbf{y}^0$ be the distinguished generators that generate $H_* (C^h)$ and $H_* (C^v)$ respectively. Observe that if $\mathbf{x}^0$ lies on $(q_0,r_0)$, then $\mathbf{y}^0$ is on $(r_0,q_0)$ and $\tau(K) = q_0 - r_0$. The generators of the unstable sequence, which have been horizontally placed, end at $(q_0-\frac{1}{2}m + \frac{1}{2}, r_0)$; and the vertical unstable sequence generators end at $(r_0, q_0 -\frac{1}{2}m + \frac{1}{2})$. Since $m=n+ 2 \tau(K)$, this implies the line $L_k$ has a nonempty intersection if $- \frac{1}{2}n \leq k \leq \frac{1}{2}n -1$.
\end{proof}

\begin{figure}
\begin{center}
\includegraphics[scale=0.6]{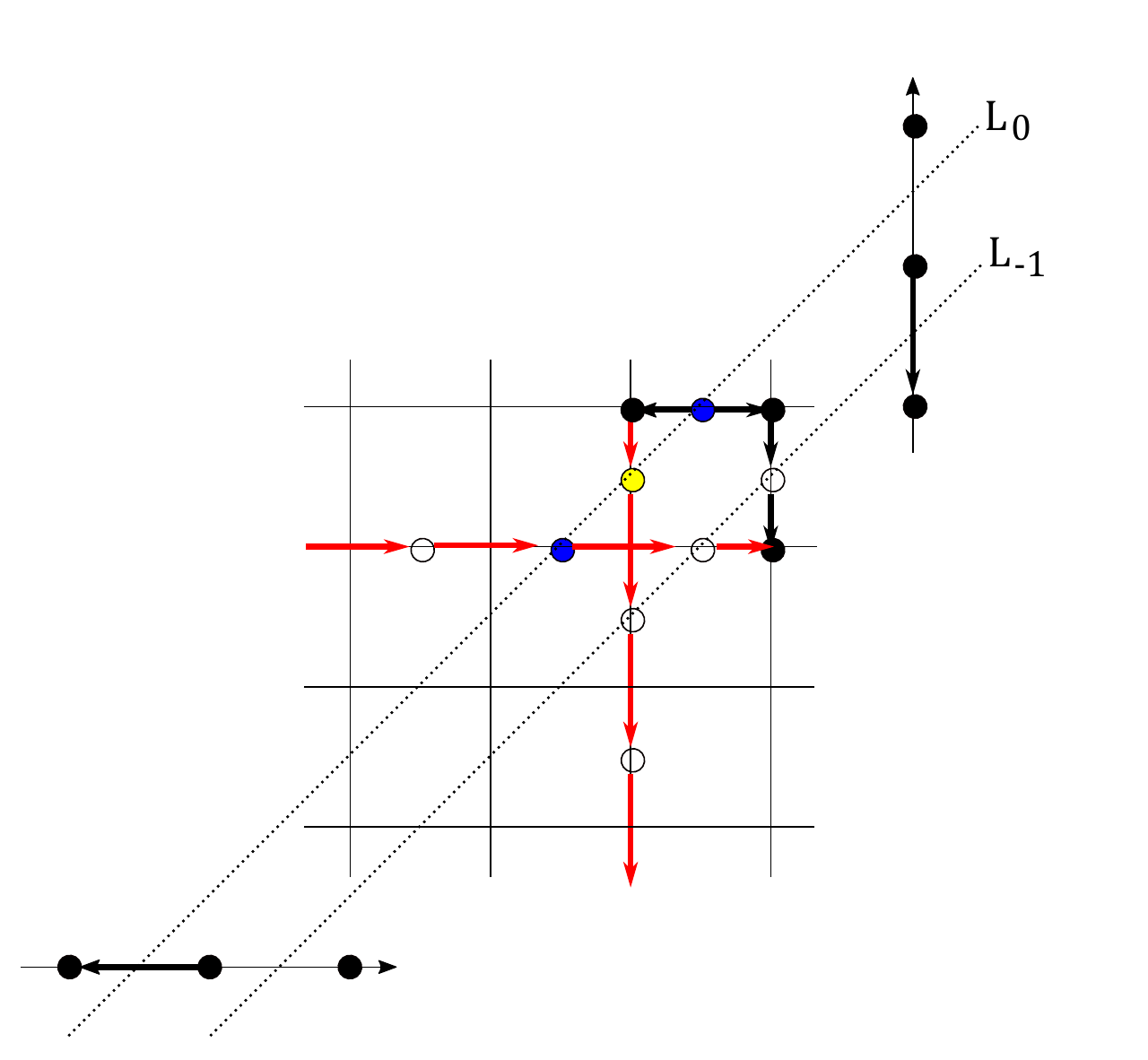}
\caption{The figure illustrates the type-$D$ module of the meridian complement in large integral surgery manifold along the right-handed trefoil. An intersection with $L_k$ and the horizontal sequence of the graph (blue dot) corresponds to a generator in the quotient complex $C^h / \mathcal{F}(K, -k-1)$ (on the bottom left line, dots on the right of the dashed line $L_k$).  Also, an intersection point in the vertical sequence (yellow dot) corresponds to a generator in $C^v / \mathcal{F}(K,k)$ (on the top right line, dots above the dashed line $L_k$). }
\end{center}
\end{figure}

\section{Example: Poincare sphere and meridional class}
\label{sec:example}

Recall that an integral homology sphere is obtained by a (sequence of) knot with surgery coefficient $\pm 1$. Hence, we first need to extend Theorem~\ref{thm:main} to arbitrary framing.

\begin{prop}
Let $CFK^-(K)$ be a reduced chain complex of a knot $K$ in $S^3$. For an arbitrary integer $n$, the type-$DD$ module of $S^3 \backslash \mathcal{L}_K$ with framing $n$ can be derived from the algorithm in Theorem~\ref{thm:main} except for the unstable chain. The unstable chain has the following structure, depending on the framing $n$.
\begin{itemize}
  \item If $n < 2 \tau(K)$
  \begin{displaymath}
  \xymatrix{
\mathbf{x}^0_{\infty} & & \mathbf{x}^0_0 \ar[ll]_{\rho_1 \sigma_3 + \rho_{123} \sigma_{123} } \ar@/^/[r]^{\rho_3} & \gamma_1 \ar@/^/[l]^{\rho_2 \sigma_{12} } \ar@/^/[r]^{\rho_{23}} & \cdots \ar@/^/[l]^{ \sigma_{12} } \ar@/^/[r]^{\rho_{23}} & \gamma_m \ar@/^/[l]^{ \sigma_{12} } \ar@/^/[r]^{\rho_{23} \sigma_1 } & \mathbf{y}^0_{\infty} \ar@/^/[l]^{\sigma_2} & & \mathbf{y}^0_0, \ar[ll]_{\rho_1 \sigma_3 + \rho_{123} \sigma_{123} }
}
  \end{displaymath}
  where $m = 2 \tau(K) -n$. 
  \item If $n = 2 \tau(K)$
  \begin{displaymath}
  \xymatrix{
  \mathbf{x}^0_{\infty} & & \mathbf{x}^0_0 \ar[ll]_{\rho_1 \sigma_3 + \rho_{123} \sigma_{123} } \ar@/^/[r]^{\rho_3 \sigma_1} & \mathbf{y}^0_{\infty} \ar@/^/[l]^{\rho_2 \sigma_2} & & \mathbf{y}^0_0. \ar[ll]_{\rho_1 \sigma_3 + \rho_{123} \sigma_{123} }
  }
  \end{displaymath}
  \item If $n > 2 \tau(K)$
  \begin{displaymath}
  \xymatrix{
\mathbf{x}^0_{\infty} & & \mathbf{x}^0_0 \ar[ll]_{\rho_1 \sigma_3 + \rho_{123} \sigma_{123} } \ar@/^/[r]^{\rho_3 \sigma_{12}} & \gamma_1 \ar@/^/[l]^{\rho_2} \ar@/^/[r]^{\sigma_{12}} & \cdots \ar@/^/[l]^{ \rho_{23} } \ar@/^/[r]^{\sigma_{12}} & \gamma_m \ar@/^/[l]^{ \rho_{23} } \ar@/^/[r]^{\sigma_1 } & \mathbf{y}^0_{\infty} \ar@/^/[l]^{\rho_{23} \sigma_2} & & \mathbf{y}^0_0, \ar[ll]_{\rho_1 \sigma_3 + \rho_{123} \sigma_{123} }
}
  \end{displaymath}
  where $m = n- 2 \tau (K)$.
\end{itemize}
\end{prop}
\begin{proof}
The statement is easily proved by taking the derived tensor product of the unstable chain of a sufficiently large negative framing parameter with the type-$DA$ module associated to $\tau_{+1}$~\cite[Figure A.3]{LOT08}.
\end{proof}
Recall that the Poincare sphere can be obtained by the Dehn surgery of the left-handed trefoil $T$ with the framing parameter $-1$. In order to obtain the type-$D$ module of the meridian complement in the Poincare sphere, first we need to have the type-$DD$ module of $T$ and its meridian complement with the framing parameter $n=-1$ (note that $\tau(T)=-1$). Its structure is written below. 
\begin{displaymath}
\xymatrix{
 & z_0 \ar[dl]_{\rho_1 \sigma_3 + \rho_{123} \sigma_{123}} \ar@/^1pc/[drr]^{\rho_3 \sigma_{12}} & & & \\
z_{\infty} \ar[d]^{\sigma_2} & & & \gamma_1 \ar@/^1pc/[llu]^{\rho_2} \ar@/^1pc/[rdd]^{\sigma_1} \\
z_1 \ar[d]^{\sigma_1} \ar@/^1pc/[u]^{\rho_{23} \sigma_1} & & & & \\
x_{\infty} \ar@/^1pc/[u]^{\rho_{23} \sigma_2 } & & & & y_{\infty} \ar@/^1pc/[uul]^{\rho_{23} \sigma_2} \\
& x_0 \ar[ul]^{\rho_1 \sigma_3 + \rho_{123} \sigma_{123}} \ar@/_1pc/[r]_{\rho_3 \sigma_{12}} & y_{-1} \ar[l]_{\rho_2} \ar@/_1pc/[r]_{\rho_2 \sigma_{12}} & y_0 \ar[l]_{\rho_3} \ar[ur]_{\rho_1 \sigma_3 + \rho_{123} \sigma_{123}} &
}
\end{displaymath}
Taking the type-$A$ module $\widehat{CFA}(\mathcal{H}_0)$ of zero surgery produces the type-$D$ module of the meridian complement in the Poincare sphere as follows.
\begin{displaymath}
\xymatrix{
z_{\infty} \ar[d]_{\sigma_2} & & \gamma_1 \ar[ll]_{\sigma_3} \ar[dd]^{\sigma_1} \\
z_1 \ar[d]_{\sigma_1} & & \\
x_{\infty} & y_{-1} \ar[l]^{\sigma_3} \ar[r]_{\sigma_{123}} & y_{\infty}
}
\end{displaymath}
Due to the bijection between generators of $\widehat{CFK}(S^3_{-1}(T), \mu_T)$ and $\jmath_2 \widehat{CFD}(S^3_{-1}(T) \backslash \mu_T)$, we can assume $\gamma_1, z_1$ and $y_1$ generate $\widehat{CFK}(S^3_{-1}(T), \mu_T)$. On the other hand, the type-$D$ module of $\mu_T$ complement in $S^3_{-1}(T)$ can be computed from the doubly pointed Heegaard diagram representing $\mu_T$ in $S^3_{-1}(T)$ by attaching the winding region near the two points $z$ and $w$. Recall that the grading is a homotopy invariance and is independent from the choice of the bordered Heegaard diagram. Then, it is clear that the sequence
\begin{displaymath}
\xymatrix{
z_1 & z_{\infty} \ar[l]_{\sigma_2} & \gamma_1 \ar[l]_{\sigma_3}
}
\end{displaymath}
exists only when there is a domain connecting from $\gamma_1$ to $z_1$ with $n_w =1$. Likewise, the sequence 
\begin{displaymath}
\xymatrix{
\gamma_1 \ar[r]^{\sigma_1} & y_{\infty} & y_{-1} \ar[l]_{\sigma_{123}}
}
\end{displaymath}
exists only when there is a domain connecting from $\gamma_1$ to $y_{-1}$ with $n_z =1$. This proves Corollary~\ref{cor}.

\end{document}